\documentclass{amsart}
\usepackage{amsfonts,amssymb,latexsym}
\usepackage{amsmath,amsthm}
\usepackage{eucal}
\usepackage[latin1]{inputenc}
\usepackage{color}

\numberwithin{equation}{section}

\newtheorem{theorem}{Theorem}[section]
\newtheorem{lemma}{Lemma}[section]
\newtheorem{proposition}{Proposition}[section]

\newtheorem{definition}{Definition}[section]
\newtheorem{remark}{Remark}[section]

\def\R{{\mathbb R}}
\def\D{{\mathbb D}}

\newcommand{\T}{\mathbb{T}}

\def\supp{\mathop{\rm supp}\nolimits}

\numberwithin{equation}{section}

\begin{document}

\title[Unconditional uniqueness for mKdV]{Unconditional uniqueness for the modified Korteweg-de Vries equation on the line }


\author[L. Molinet, D. Pilod and S. Vento]{Luc Molinet$^*$, Didier Pilod$^\dagger$  and St\'ephane Vento$^*$}

\thanks{$^*$ Partially supported by the french ANR project GEODISP}
\thanks{$^{\dagger}$ Partially supported by CNPq/Brazil, grants 302632/2013-1 and 481715/2012-6.}

\address{Luc Molinet, L.M.P.T., CNRS UMR 7350, F\'ed\'eration Denis Poisson-CNRS, Universit\'e Fran\c cois Rabelais, Tours, Parc Grandmont, 37200 Tours, France.}
\email{Luc.Molinet@lmpt.univ-tours.fr}

\address{Didier Pilod, Instituto de Matem\'atica, Universidade Federal do Rio de Janeiro, Caixa Postal 68530, CEP: 21945-970, Rio de Janeiro, RJ, Brasil.}
\email{didier@im.ufrj.br}

\address{St\'ephane Vento, Universit\'e Paris 13, Sorbonne Paris Cit\'e, LAGA, CNRS ( UMR 7539),  99, avenue Jean-Baptiste Cl\'ement, F-93 430 Villetaneuse, France.}
\email{vento@math.univ-paris13.fr}

\date{\today}

\vspace{-0.5cm}

\begin{abstract}
We prove that the modified Korteweg- de Vries equation (mKdV) equation is unconditionally well-posed in $H^s(\mathbb R)$ for $s> \frac 13$. Our
 method of proof combines the improvement of the energy method introduced recently by the first and third authors with the construction of a modified energy. Our approach
  also yields \textit{a priori} estimates for the solutions of mKdV in $H^s(\mathbb R)$, for $s>0$, and enables us to construct weak solutions at this level of regularity.

\end{abstract}

\maketitle

\section{Introduction}

We consider the initial value problem (IVP) associated to the modified Korteweg-de Vries (mKdV) equation
\begin{equation} \label{mKdV}
\left\{ \begin{array}{l}\partial_tu+\partial_x^3u+\kappa\partial_x(u^3)=0 \, , \\
u(\cdot,0)=u_0 \, , \end{array} \right.
\end{equation}
where  $u=u(x,t)$ is a real function, $\kappa=1$ or $-1$, $x \in \mathbb R$, $t \in \mathbb R$.

In the seminal paper \cite{KPV2}, Kenig, Ponce and Vega proved the well-posedness of \eqref{mKdV} in $ H^s(\R) $ for $ s\ge 1/4 $.
This result is sharp on the $ H^s$-scale in the sense that the flow map associated to mKdV fails to be uniformly continuous in $H^s(\mathbb R)$ if $s<\frac14$ in both the focusing case $\kappa=1$ (cf. Kenig, Ponce and Vega \cite{KPV3}) and the defocusing case  $\kappa=-1$ (cf. Christ, Colliander and Tao \cite{ChCoTao}).
Global well-posedness (GWP) for mKdV was proved in $H^s(\mathbb R)$ for $s>\frac14$ by Colliander, Keel, Staffilani, Takaoka and Tao \cite{CKSTT} by using the $I$-method (see also \cite{Guo,Kish}  for the GWP at the end point $s=1/4$). We also mention that another proof of the local well-posedness result for $s \ge \frac14$ was given by Tao by using the Fourier restriction norm method \cite{Tao}. On the other hand, if one exits the $ H^s $-scale, Gr\"unrock \cite{G} and then Gr\"unrock-Vega \cite{GV} proved that the Cauchy problem is well-posed in $ \widehat{H^r_s} $ for  $1<r<2$ and $ s\ge \frac{1}{2}-\frac{1}{2r} $ where $ \|u_0\|_{\widehat{H^r_s}} :=\|\langle \xi \rangle^s \widehat{u_0} \|_{L^{r'}_\xi} $
with $ \frac{1}{r'}+\frac{1}{r}=1 $. Note that  $ \widehat{H^1_0} $ is critical for scaling considerations and thus the result in \cite{GV} is nearly optimal  on this scale whereas the index $ 1/4 $ in the $ H^s $-scale is far above the critical index which is $ -1/2$.

\medskip

The proof of the well-posedness result in \cite{KPV2} relies on the dispersive estimates associated with the linear group of \eqref{mKdV}, namely the Strichartz estimates, the local smoothing effect and the maximal function estimate. A normed function space is constructed based on those estimates and allows to solve \eqref{mKdV} via a fixed point theorem on the associated integral equation. Of course the solutions obtained in this way are unique in this
 resolution  space.  The same occurs for the solutions constructed by Tao which are unique in the space $X_T^{s,\frac12+} $.

The question to know whether uniqueness holds for solutions which do not belong to these resolution spaces turns out to be far from trivial at this level of regularity. This kind of question was first raised by Kato \cite{Ka} in the Schr\"odinger equation context. We refer to such uniqueness in $C([0, T] : H^s(\mathbb R))$, or more generally in $ L^\infty(]0,T[ : H^{s}(\mathbb R))$, without intersecting with any auxiliary function space as \textit{unconditional uniqueness}.  This ensures  the uniqueness of the weak solutions to  the equation at the $ H^s$-regularity. This is useful, for instance, to pass to the limit on perturbations of the equation  as the perturbative coefficient tends to zero (see for instance \cite{M} for such an application).

Unconditional uniqueness was proved for the KdV equation to hold in $L^2(\mathbb R)$ \cite{Zhou} and in $L^2(\mathbb T)$ \cite{BaIlTi} and for the mKdV in $H^{\frac12}(\mathbb T)$ \cite{KwonOh}.

\medskip
The aim of this paper is to propose a  strategy to show the unconditional uniqueness for some dispersive PDEs and, in particular, to prove the unconditional uniqueness of  the mKdV equation in $H^s(\mathbb R)$ for $s > \frac13 $. Note that, doing so, we also provide a different proof of the existence result.  Before stating our main result, we give a precise definition of our notion of solution.
\begin{definition}\label{def} Let $T>0$ . We will say that $u\in L^3(]0,T[\times \R) $ is a solution to \eqref{mKdV} associated with the initial datum $ u_0 \in H^s(\R)$  if
  $ u $ satisfies \eqref{mKdV}   in the distributional sense, i.e. for any test function $ \phi\in C_c^\infty(]-T,T[\times \R) $,  there holds
  \begin{equation}\label{weakmKdV}
  \int_0^\infty \int_{\R} \Bigl[(\phi_t +\partial_x^3 \phi )u +  \phi_x u^3\Bigr] \, dx \, dt +\int_{\R} \phi(0,\cdot) u_0 \, dx =0\; .
  \end{equation}
 \end{definition}
 \begin{remark} \label{rem2} Note that $ L^\infty(]0,T[ \, : H^s(\R)) \hookrightarrow  L^3(]0,T[\times \R) $ as soon as $s \ge 1/6$.  Moreover, for $u\in  L^\infty(]0,T[ \, : H^s(\R))$,  with $ s\ge \frac{1}{6} $, $ u^3 $ is well-defined and
  belongs to $ L^\infty(]0,T[ : L^1(\R))$.   Therefore \eqref{weakmKdV} forces $ u_t \in L^\infty(]0,T[ \, : H^{-3}(\R)) $  and  ensures that  \eqref{mKdV}
   is satisfied in $ L^\infty(]0,T[ \, : H^{-3}(\R)) $.
  In particular, $ u\in C([0,T] : H^{-3}(\R))$ and \eqref{weakmKdV}  forces  the initial condition $ u(0)=u_0 $. Note that , since $ u\in L^\infty(]0,T[ \, : H^s(\R)) $, this actually ensures that  $ u\in C_w([0,T] : H^{s}(\R))$ and that $ u\in C([0,T] : H^{s'}(\R))$ for any $ s'<s $. Finally, we notice that this also ensures that $ u $ satisfies the Duhamel formula associated with \eqref{mKdV} in $C([0,T] : H^{-3}(\mathbb R))$.
  \end{remark}

\begin{theorem} \label{maintheo}
Let $s > 1/3$ be given. \mbox{ } \\
\underline{\it Existence :}
 For any $u_0 \in H^s(\mathbb R)$, there exists $T=T(\|u_0\|_{H^s}) >0$ and a  solution $u$ of the IVP \eqref{mKdV} such that
\begin{equation} \label{maintheo.1}
u \in C([0,T] : H^s(\mathbb R)) \cap L^4_TL^{\infty}_x \cap X^{s-1,1}_T \cap X^{s-\frac{7}{8},\frac{15}{16}}_T \, .
\end{equation}

\noindent \underline{\it Uniqueness :} The solution is unique in the class
\begin{equation} \label{maintheo.1b}
u\in L^\infty(]0,T[ \, : H^{s}(\mathbb R)) \, .
\end{equation}

Moreover, the flow map
data-solution $:u_0 \mapsto u$ is Lipschitz from $H^s(\mathbb R)$
into $C([0,T] : H^s(\mathbb R))$.
\end{theorem}
\begin{remark}
We refer to Section 2.2 for the definition of the norms  $\|u\|_{X^{s,b}_T}$.
\end{remark}

\medskip
Our technique of proof also yields \textit{a priori} estimates for the solutions of mKdV in $H^s(\mathbb R)$ for $s>0$. It is worth noting that \textit{a priori} estimates in $H^s(\mathbb R)$ were already proved by Christ, Holmer and Tataru for $-\frac18<s<\frac14$ in \cite{ChHoTa}. Their proof relies on short time Fourier restriction norm method in the context of the atomic spaces $U$, $V$ and the $I$-method. Although our result is not as strong as Christ, Holmer and Tataru's one, we hope that it still may be of interest due to the simplicity of our proof.

\begin{theorem} \label{secondtheo}
Let $s >0 $ and $ u_0\in H^\infty(\R) $. Then there exists $T=T(\|u_0\|_{H^s})>0$  such that the solution $ u $ to \eqref{mKdV} emanating from $ u_0 $ satisfies\footnote{See Section \ref{spaces} for the definition of the $ \widetilde{L^{\infty}_T}H^s_x$-norm.}
\begin{equation} \label{secondtheo.1}
\|u\|_{\widetilde{L^{\infty}_T}H^s_x}+\|u\|_{X^{s-1,1}_T}+\|u\|_{L^4_TL^{\infty}_x} \lesssim \|u_0\|_{H^s_x} \, .
\end{equation}
Moreover, for any $u_0\in H^s(\R) $, there exists a solution $u\in L^\infty_T H^s_x\cap L^4_T L^\infty_x $ to \eqref{mKdV} emanating from $u_0$ that satisfies
\eqref{secondtheo.1}.
\end{theorem}
\medskip
\begin{remark}
Note that for $ u_0\in L^2(\R) $, the existence of weak solutions of \eqref{mKdV}, in the sense of Definition \ref{def}, is well-known by making use of the  so-called Kato smoothing effect. Such solution belongs to $L^\infty_t L^2_x \cap L^2_{t,loc} H^1_{loc} $. Our result indicates that if $ u_0$  belongs to $ H^s(\R) $, $s>0$, instead of $ L^2(\R) $, then we can ask the weak solution to satisfy also \eqref{secondtheo.1} and, in particular, to propagate the $ H^s $-regularity on some time interval.
\end{remark}
\medskip

To prove Theorems \ref{maintheo} and  \ref{secondtheo}, we derive energy estimates on the dyadic blocks $ \|P_Nu\|_{H^s_x}^2$ by  taking advantage of the resonant relation and the fact that any solution enjoys some conormal regularity.
 This approach has been introduced by the first and the third authors in \cite{MoVe}. Note however that, here, to bound some Bourgain's norm of a solution,  we  need first  to bound its $\|\cdot\|_{L^4_TL^{\infty}_x}$-norm. This norm is in turn controlled by using a refined Strichartz estimate derived by chopping the time interval in small pieces whose length depends on the spatial frequency.  Note that it was first established by Koch and Tzvetkov \cite{KoTz} (see also Kenig and Koenig \cite{KeKo} for an improved version) in the Benjamin-Ono context.

The main difficulty to estimate $\frac{d}{dt}\|P_Nu\|_{H^s_x}^2$ is to handle the resonant term $\mathcal{R}_N$, typical of the cubic nonlinearity $\partial_x(u^3)$.  When $u$ is the solution of mKdV, $\mathcal{R}_N$ writes $\mathcal{R}_N=\int \partial_x\big( P_{+N}uP_{+N}uP_{-N}u\big)P_{-N}u dx$. Actually, it turns out that we can always put the derivative appearing in $\mathcal{R}_N$ on a low frequency product by integrating by parts$\footnote{For technical reason we perform this integration by parts in Fourier variables.}$, as it was done in \cite{IoKeTa} for quadratic nonlinearities. This allows us to derive the \textit{a priori} estimate of Theorem \ref{secondtheo} in $H^s(\mathbb R)$ for $s>0$. Unfortunately, this is not the case anymore for the difference of two solutions of mKdV due to the lack of symmetry of the corresponding equation. To overcome this difficulty we modify the $H^s$-norm by higher order terms up to order 6. These higher order terms are constructed so that the  contribution  of their time derivatives coming from the linear part of the equation will cancel out the resonant term $\mathcal{R}_N$. The use of a modified energy is well-known to be a  quite powerful  tool in PDE's (see for instance \cite{MN} and \cite{KePi}).
 Note however that, in our case, we need to define the modified energy in Fourier variables due to the resonance relation associated to the cubic nonlinearity. This way to construct the modified energy has much in common with  the way to construct the modified energy in the I-method (cf. \cite{CKSTT}).
\smallskip

Finally let us mention that the tools developed in this paper together with some ideas of \cite{TT} and \cite{NTT} enabled us in \cite{MoPiVe} to get the unconditional well-posedness of the periodic mKdV equation in $ H^s(\T) $ for $ s\ge 1/3$.
We also  hope that the techniques introduced here could be useful in the study of the Cauchy problem at low regularity of other cubic nonlinear dispersive equations such as the modified Benjamin-Ono equation and the derivative nonlinear Schr\"odinger equation.

\medskip
The rest of the paper is organized as follows. In Section \ref{notation}, we introduce the notations, define the function spaces and state some preliminary estimates. The multilinear estimates at the $L^2$-level are proved in Section \ref{Secmultest}. Those estimates are used to derive the energy estimates in Section \ref{Secenergy}. Finally, we give the proofs of Theorems \ref{maintheo} and \ref{secondtheo} respectively in Sections \ref{Secmaintheo} and \ref{Secsecondtheo}.

\section{Notation, Function spaces and preliminary estimates} \label{notation}

\subsection{Notation}\label{subnotation}
For any positive numbers $a$ and $b$, the notation $a \lesssim b$
means that there exists a positive constant $c$ such that $a \le c
b$. We also denote $a \sim b$ when $a \lesssim b$ and $b \lesssim
a$. Moreover, if $\alpha \in \mathbb R$, $\alpha_+$, respectively
$\alpha_-$, will denote a number slightly greater, respectively
lesser, than $\alpha$.

Let us denote by  $\mathbb D =\{N>0 : N=2^n \ \text{for some} \ n \in \mathbb Z \}$ the dyadic numbers. Usually, we use $n_i$, $j_i$, $m_i$ to denote integers and $N_i=2^{n_i}$, $L_i=2^{j_i}$ and $M_i=2^{m_i}$ to denote dyadic numbers.

For $N_1, \ N_2 \in \mathbb D$, we use the notation $N_1 \vee N_2=\max\{N_1,N_2\}$ and $N_1 \wedge N_2 =\min\{N_1,N_2\}$.
Moreover, if $N_1, \, N_2, \, N_3 \in \mathbb D$, we also denote by $N_{max} \ge N_{med} \ge N_{min}$ the maximum, sub-maximum and minimum of $\{N_1,N_2,N_3\}$.

For $u=u(x,t) \in \mathcal{S}'(\mathbb R^2)$,
$\mathcal{F}u$ will denote its space-time Fourier
transform, whereas $\mathcal{F}_xu=\widehat{u}$, respectively
$\mathcal{F}_tu$, will denote its Fourier transform
in space, respectively in time. For $s \in \mathbb R$, we define the
Bessel and Riesz potentials of order $-s$, $J^s_x$ and $D_x^s$, by
\begin{displaymath}
J^s_xu=\mathcal{F}^{-1}_x\big((1+|\xi|^2)^{\frac{s}{2}}
\mathcal{F}_xu\big) \quad \text{and} \quad
D^s_xu=\mathcal{F}^{-1}_x\big(|\xi|^s \mathcal{F}_xu\big).
\end{displaymath}

We also denote by $U(t)=e^{-t\partial_x^3}$ the unitary group associated to the linear part of \eqref{mKdV}, \textit{i.e.},
\begin{displaymath}
U(t)u_0=e^{-t\partial_x^3}u_0=\mathcal{F}_x^{-1}\big(e^{it\xi^3}\mathcal{F}_x(u_0)(\xi) \big) \, .
\end{displaymath}

Throughout the paper, we fix a smooth cutoff function $\chi$ such that
\begin{displaymath}
\chi \in C_0^{\infty}(\mathbb R), \quad 0 \le \chi \le 1, \quad
\chi_{|_{[-1,1]}}=1 \quad \mbox{and} \quad  \mbox{supp}(\chi)
\subset [-2,2].
\end{displaymath}
We set  $ \phi(\xi):=\chi(\xi)-\chi(2\xi) $. For $l \in \mathbb Z$, we define
\begin{displaymath}
\phi_{2^l}(\xi):=\phi(2^{-l}\xi), \end{displaymath} and, for $ l\in \mathbb Z \cap [1,+\infty) $,
\begin{displaymath}
\psi_{2^{l}}(\xi,\tau)=\phi_{2^{l}}(\tau-\xi^3).
\end{displaymath}
By convention, we also denote
\begin{displaymath}
\phi_0(\xi)=\chi(2\xi) \quad \text{and} \quad \psi_{0}(\xi,\tau):=\chi(2(\tau-\xi^3)) \, .
\end{displaymath}
Any summations over capitalized variables such as $N, \, L$, $K$ or
$M$ are presumed to be dyadic. Unless stated otherwise, we will work with non-homogeneous dyadic decompositions in $N$, $L$ and $K$,
\textit{i.e.} these variables range over numbers of the form $\mathbb D_{nh}=\{2^k
: k \in \mathbb N \} \cup \{0\}$, whereas we will work with homogeneous dyadic decomposition in $M$, \textit{i.e.} these variables range over $\mathbb D$ . We call the numbers in $\mathbb D_{nh}$ \textit{nonhomogeneous dyadic numbers}. Then, we have that $\displaystyle{\sum_{N}\phi_N(\xi)=1}$,
\begin{displaymath}
 \mbox{supp} \, (\phi_N) \subset
I_N:=\{\frac{N}{2}\le |\xi| \le 2N\}, \ N \ge 1, \quad \text{and} \quad
\mbox{supp} \, (\phi_0) \subset I_0:=\{|\xi| \le 1\}.
\end{displaymath}

Finally, let us define the Littlewood-Paley multipliers $P_N$, $R_K$ and $Q_L$ by
\begin{displaymath}
P_Nu=\mathcal{F}^{-1}_x\big(\phi_N\mathcal{F}_xu\big), \quad R_Ku=\mathcal{F}^{-1}_t\big(\phi_K\mathcal{F}_tu\big) \quad \text{and} \quad
Q_Lu=\mathcal{F}^{-1}\big(\psi_L\mathcal{F}u\big),
\end{displaymath}
 $P_{\ge N}:=\sum_{K \ge N} P_{K}$,  $P_{\le N}:=\sum_{K \le N} P_{K}$, $Q_{\ge L}:=\sum_{K \ge L} Q_{K}$ and   $Q_{\le L}:=\sum_{K \le L} Q_{K}$.

 Sometimes, for the sake of simplicity and when there is no risk of confusion, we also denote $u_N=P_Nu$.

\subsection{Function spaces} \label{spaces}
For $1 \le p \le \infty$, $L^p(\mathbb R)$ is the usual Lebesgue
space with the norm $\|\cdot\|_{L^p}$. For $s \in \mathbb R$,
the Sobolev space $H^s(\mathbb R)$  denotes the space of all  distributions of $\mathcal{S}'(\mathbb R)$ whose usual
norm $\|u\|_{H^s}=\|J^s_xu\|_{L^2}$ is finite.

 If $B$ is one of the spaces defined above, $1 \le p \le \infty$ and $T>0$, we
define the space-time spaces $L^p_ t B_x$, $L^p_TB_x$, $ \widetilde{L^p_t} B_x $ and $\widetilde{L^p_T}B_x$
 equipped with  the norms
\begin{displaymath}
\|u\|_{L^p_ t B_x} =\Big(\int_{\R}\|f(\cdot,t)\|_{B}^pdt\Big)^{\frac1p} , \quad
\|u\|_{L^p_ T B_x} =\Big(\int_0^T\|f(\cdot,t)\|_{B}^pdt\Big)^{\frac1p}
\end{displaymath}
with obvious modifications for $ p=\infty $,  and
\begin{displaymath}
\|u\|_{\widetilde{L^p_ t }B_x} =\Big(\sum_{N}
 \| P_N u \|_{L^p_ t B_x}^2\Big)^{\frac12}, \quad \|u\|_{\widetilde{L^p_ T }B_x} =\Big(\sum_{N}
 \| P_N u \|_{L^p_ T B_x}^2\Big)^{\frac12} \, .
\end{displaymath}

For $s$, $b \in \mathbb R$, we introduce the Bourgain spaces
$X^{s,b}$ related to the linear part of \eqref{mKdV} as
the completion of the Schwartz space $\mathcal{S}(\mathbb R^2)$
under the norm
\begin{equation} \label{X1}
\|u\|_{X^{s,b}} := \left(
\int_{\mathbb{R}^2}\langle\tau-\xi^3\rangle^{2b}\langle \xi\rangle^{2s}|\mathcal{F}(u)(\xi, \tau)|^2
d\xi d\tau \right)^{\frac12},
\end{equation}
where $\langle x\rangle:=1+|x|$.
By using the definition of $U$, it is easy to see that
\begin{equation} \label{X2}
\|u\|_{X^{s,b}}\sim\| U(-t)u \|_{H^{s,b}_{x,t}} \quad \text{where} \quad \|u\|_{H^{s,b}_{x,t}}=\|J^s_xJ^b_tu\|_{L^2_{x,t}} \, .
\end{equation}

We define our resolution space $Y^s=X^{s-1,1} \cap X^{s-\frac{7}{8},\frac{15}{16}} \cap \widetilde{L^\infty_t} H^s_x $,   with the associated norm
\begin{equation} \label{YT}
\|u\|_{Y^s}=\|u\|_{X^{s-1,1}}+\|u\|_{X^{s-\frac{7}{8},\frac{15}{16}}}+\|u\|_{\widetilde{L^\infty_t}H^s_x} \; .
\end{equation}
It is clear from the definition that $\widetilde{L^{\infty}_T}H^s_x \hookrightarrow L^{\infty}_TH^s_x$, \textit{i.e.}
\begin{equation} \label{tildenorm}
\|u\|_{L^{\infty}_TH^s_x} \lesssim \|u\|_{\widetilde{L^{\infty}_T}H^s_x}, \quad \forall \, u \in \widetilde{L^{\infty}_T}H^s_x \, .
\end{equation}
Note that this estimate still holds true if we replace $T$ by $t$. However, the reverse inequality is only true if we allow a little loss in space regularity. Let $s', \, s \in \mathbb R$ be such that $s'<s$. Then,
\begin{equation} \label{tildenorm2}
\|u\|_{\widetilde{L^{\infty}_T}H^{s'}_x} \lesssim \|u\|_{L^{\infty}_TH^s_x}, \quad \forall \, u \in L^{\infty}_TH^s_x \, .
\end{equation}

Finally, we will also use a restriction in time versions of these spaces.
Let $T>0$ be a positive time and $F$ be a normed space of space-time functions. The restriction space $F_T$ will be the space of functions $u : \mathbb R \times [0,T] \longrightarrow \mathbb R$  satisfying
\begin{displaymath}
\|u\|_{F_T} =\inf \big\{ \|\tilde{u}\|_{F} \ : \ \tilde{u}: \mathbb R \times \mathbb R\to \mathbb R \ \text{and} \ \tilde{u}_{|_{\mathbb R \times [0,T]}} = u \big\} <\infty\, .
\end{displaymath}

\subsection{Extension operator} \label{extsec}
We introduce an extension operator $\rho_T$ which is a bounded operator from $  \widetilde{L^{\infty}_T}H^s_x \cap X^{s-1,1}_T\cap X^{s-\frac78,\frac{15}{16}}_T \cap L^{4}_TL^{\infty}_x$  into $\widetilde{L^{\infty}_t}H^s_x \cap X^{s-1,1} \cap X^{s-\frac78,\frac{15}{16}} \cap L^{4}_tL^{\infty}_x$.

\begin{definition} \label{def.extension}
Let $0<T \le 1$ and $u:\mathbb R \times [0,T] \rightarrow \mathbb R$ be given. We define the extension operator $\rho_T$ by
\begin{equation}\label{defrho}
\rho_T(u)(t):= U(t)\chi(t) U(-\mu_T(t)) u(\mu_T(t))\; ,
\end{equation}
where $ \chi $ is the smooth cut-off function defined in Section \ref{subnotation} and $\mu_T $ is the  conti-nuous piecewise affine function defined  by
\begin{displaymath}
 \mu_T(t)=\left\{\begin{array}{rcl}
 0  &\text{for } &  t<0 \\
 t  &\text {for }& t\in [0,T] \\
  T & \text {for } & t>T
 \end{array}
 \right. .
\end{displaymath}
\end{definition}

It is clear from the definition that $\rho_T(u)(x,t)=u(x,t)$ for $(x,t) \in \mathbb R \times [0,T]$.

\begin{lemma} \label{extension}
Let $0<T \le 1$, $s, \, \alpha, \, \theta, \,  b \in \mathbb R$ such that $\alpha \le s+\frac14$ and $\frac12<b \le 1$. Then,
\begin{displaymath}
\begin{split}
\rho_T : \ &  \widetilde{L^{\infty}_T}H^s_x \cap X^{\theta,b}_T \cap L^4_TW^{\alpha,\infty}_x \longrightarrow \widetilde{L^{\infty}_t}H^s_x \cap X^{\theta,b} \cap L^4_tW^{\alpha,\infty}_x \\
 &u \mapsto \rho_T(u)
\end{split}
\end{displaymath}
is a bounded linear operator, \textit{i.e.}
\begin{equation} \label{extension.1}
\begin{split}
\|\rho_T(u)\|_{\widetilde{L^{\infty}_t}H^s_x} + \|\rho_T(u)\|_{X^{\theta,b}}&+\|\rho_T(u)\|_{L^4_tW^{\alpha,\infty}_x} \\ &\lesssim \|u\|_{\widetilde{L^{\infty}_T}H^s_x}+\|u\|_{X^{\theta,b}_T}+\|u\|_{L^4_TW^{\alpha,\infty}_x} \, ,
\end{split}
\end{equation}
for all $u \in  \widetilde{L^{\infty}_T}H^s_x \cap X^{\theta,b}_T \cap L^4_TW^{\alpha,\infty}_x$.

Moreover, the implicit constant in \eqref{extension.1} can be chosen independent of $0<T \le 1$, $s, \, \alpha, \, \theta$ and $\frac12 < b \le 1$.
\end{lemma}

\begin{proof}
First, the unitarity of the free group $ U(\cdot) $ in $ H^s(\R) $ easily leads to
\begin{displaymath}
\|\rho_T(u)\|_{\widetilde{L^\infty_t}H^s_x} \lesssim \|u(\mu_T(\cdot))\|_{\widetilde{L^\infty_t }H^s_x}\lesssim \|u\|_{\widetilde{L^\infty_T}  H^s_x} + \|u(0)\|_{H^s}+ \|u(T)\|_{H^s} \; .
\end{displaymath}
Now, since $ b>1/2$, it is well-known (see for instance \cite{Gi}), that $X_T^{\theta,b} \hookrightarrow C([0,T]:H^{\theta}(\mathbb R))$. Therefore,
$u\in C([0,T]:H^{\theta}(\mathbb R))\cap  \widetilde{L^{\infty}_T}H^s_x
 \hookrightarrow  C([0,T]:H^{\theta}(\mathbb R))\cap L^{\infty}_TH^s_x$ and  we claim that
\begin{equation} \label{tg100}
\|u(0)\|_{H^s} \le \|u\|_{L^{\infty}_TH^s_x} \quad \text{and} \quad \|u(T)\|_{H^s} \le  \|u\|_{L^{\infty}_TH^s_x}\, .
\end{equation}
Indeed, if it is not the case, assuming for instance that $\|u(0)\|_{H^s} > \|u\|_{{L^{\infty}_T}H^s_x}$, there would exist $\epsilon>0$ and a decreasing sequence $\{t_n\} \subset (0,T)$ tending to $0$ such that for any $n \in \mathbb N$, $\|u(t_n)\|_{H^s} \le \|u(0)\|_{H^s}-\epsilon$.  The continuity of $ u $ with values in $ H^\theta(\R) $ then ensures that $ u(t_n) \rightharpoonup u(0) $ in $H^s(\mathbb R)$, which forces $ \|u(0)\|_{H^s}\le \liminf \|u(t_n)\|_{H^s} $   and yields the contradiction. Therefore, we conclude by using \eqref{tildenorm} that
\begin{equation} \label{tg1}
\|\rho_T(u)\|_{\widetilde{L^\infty_t}H^s_x} \lesssim  \|u\|_{\widetilde{L^\infty_T}  H^s_x} \, .
\end{equation}

Second, according to classical results on extension operators (see for instance \cite{LM}),  for any $ 1/2<b\le 1$,  $f \mapsto \chi f(\mu_T(\cdot)) $ is  linear continuous    from $ H^b([0,T]) $ into $ H^b(\R) $  with a bound that does not depend on $ T>0 $. Then, 
the definition of the $ X^{\theta,b}$-norm leads, for $1/2<b\le 1 $ and   $\theta\in \R $, to
\begin{equation}\label{tg2}
\|\rho_T(u)\|_{X^{\theta,b}}  = \|\chi\, U(-\mu_T(\cdot)) u(\mu_T(\cdot))\|_{H^{{\theta},b}_{x,t} }
 \lesssim  \|U(-\cdot) u\|_{H^b([0,T[; H^{\theta})}
 \lesssim  \|u\|_{X^{\theta,b}_T} \; .
\end{equation}

Finally,  for $ \alpha \in \R $,
\begin{align*}
\|J^\alpha_x \rho_T(u)\|_{L^4_t L^\infty_x} & \lesssim  \|\chi U(-\cdot)J^\alpha_x   u(0)\|_{L^4(]-\infty,0[; L^\infty_x)} + \|J^\alpha_x u \|_{L^4_T L^\infty_x}
 \\
 & + \|\chi U(-\cdot)J^\alpha_x  U(T) u(T)\|_{L^4(]T,+\infty[; L^\infty_x)} \,.
\end{align*}
Now by using the Strichartz estimate related to the unitary group $U$ (see estimate \eqref{strichartz} in the next subsection), we deduce that
\begin{align*}
  \|\chi U(-\cdot)J^\alpha_x   u(0)\|_{L^4(]-\infty,0[; L^\infty_x)} & \lesssim  \|   U(-\cdot)J^\alpha_x   u(0)\|_{L^4(]-2,0[;L^\infty_x)}\lesssim \|u(0)\|_{H^s_x} \, ,
 \end{align*}
 since $\alpha \le s-\frac14$, and in the same way
 \begin{align*}
  \|\chi U(-\cdot) U(T)J^\alpha_x   u(T)\|_{L^4(]T,+\infty[; L^\infty_x)}
 & \lesssim \|U(T)u(T)\|_{H^s}=\|u(T)\|_{H^{s}}\; .
 \end{align*}
This ensures by using \eqref{tg100} that
\begin{equation} \label{tg3}
\|J^\alpha_x \rho_T(u)\|_{L^4_t L^\infty_x} \lesssim \|J^\alpha_x u \|_{L^4_T L^\infty_x} + \|u\|_{\widetilde{L^\infty_T}  H^s_x} \, .
\end{equation}

Therefore, we conclude the proof of \eqref{extension.1} gathering \eqref{tg1}-\eqref{tg3}.
\end{proof}

\begin{remark}
In the following, we will work with the resolution space $Y^s_T$. While it follows clearly from the definition of $Y^s_T$ that
\begin{equation} \label{resolution.1}
\|u\|_{\widetilde{L^{\infty}_T}H^s_x}+\|u\|_{X^{s-1,1}_T}+\|u\|_{X^{s-\frac78,\frac{15}{16}}_T} \lesssim \|u\|_{Y^s_T}, \quad \forall \, u \in Y^s_T \, ,
\end{equation}
the reverse inequality is not straightforward. However, it can be proved by using the extension operator $\rho_T$. Indeed, it follows from  Lemma \ref{extension} that
\begin{equation} \label{resolution.2}
\|u\|_{Y^s_T} \le \|\rho_T(u)\|_{Y^s} \lesssim \|u\|_{\widetilde{L^{\infty}_T}H^s_x}+\|u\|_{X^{s-1,1}_T}+\|u\|_{X^{s-\frac78,\frac{15}{16}}_T}\, .
\end{equation}
In particular, this proves that $Y^s_T= \widetilde{L^{\infty}_T}H^s_x \cap X^{s-1,1}_T \cap X^{s-\frac78,\frac{15}{16}}_T$.
\end{remark}

\subsection{Refined Strichartz estimates} First, we recall the Strichartz estimates associated to the unitary Airy group derived in \cite{KPV1}. For all
 $u_0\in L^2(\mathbb R)$
\begin{equation} \label{strichartz}
\|e^{-t\partial^3_x}D_x^{\frac14}u_0 \|_{L^4_tL^{\infty}_x} \lesssim \|u_0\|_{L^2} \ ,
\end{equation}
and for all $ g\in  L^{\frac43}_t L^1_x$,
\begin{equation} \label{strichartz2}
\Bigr\|\int_0^t e^{-(t-t')\partial^3_x}D_x^{\frac12} g(t')\, dt' \Bigl\|_{L^4_tL^{\infty}_x} \lesssim \|g\|_{L^{\frac43}_t L^1_x} \ .
\end{equation}
Note that these two estimates are equivalent thanks to the $ T T^* $-argument. \\
Following the arguments in \cite{KeKo} and \cite{KoTz}, we derive a refined Strichartz estimate for the solutions of the linear problem
\begin{equation} \label{linearKdV}
\partial_tu+\partial_x^3u=F \, .
\end{equation}

\begin{proposition} \label{refinedStrichartz}
Assume that $T>0$ and $\delta \ge 0$. Let $u$ be a smooth solution to \eqref{linearKdV} defined on the time interval $[0,T]$. Then,
\begin{equation} \label{refinedStrichartz1}
\|u\|_{L^4_TL^{\infty}_x} \lesssim \|J_x^{\frac{\delta-1}4+\theta}u\|_{L^{\infty}_TL^2_x}+\|J_x^{-\frac{\delta+1}2+\theta}F\|_{L^4_TL^1_x}  \ ,
\end{equation}
for any $\theta>0$.
\end{proposition}

\begin{proof} Let $u$ be solution to \eqref{linearKdV} defined on a time interval $[0,T]$. We use a nonhomogeneous Littlewood-Paley
decomposition, $u=\sum_Nu_N$ where $u_N=P_Nu$, $N$ is a nonhomogeneous dyadic number and also denote $F_N=P_NF$. Then, we get from the Minkowski  inequality that
\begin{displaymath}
\|u\|_{L^4_TL^{\infty}_x}\le  \sum_N\|u_N\|_{L^4_TL^{\infty}_x}
\lesssim \sup_{N}N^{\theta}\|u_N\|_{L^4_TL^{\infty}_x} \, ,
\end{displaymath}
for any $\theta>0$. Recall that $P_0$ corresponds to the projection in low frequencies, so that we set $0^{\theta}=1$ by convention. Since H\"older and Bernstein inequalities easily yield
$$
\|P_0 u \|_{L^4_TL^{\infty}_x}\lesssim  T^{\frac14} \| P_0 u \|_{L^\infty_T L^2_x} \; ,
$$
it is enough to prove that
\begin{equation} \label{refinedStrichartz2}
\|u_N\|_{L^4_TL^{\infty}_x} \lesssim \|D_x^{\frac{\delta-1}4}u_N\|_{L^{\infty}_TL^2_x}+\|D_x^{-\frac{\delta+1}2}F_N\|_{L^4_TL^1_x} \, ,
\end{equation}
for any $\delta \ge 0$ and any dyadic number $N \in \{2^k : k \in \mathbb N\} $.

Let $\delta$ be a nonnegative number. We chop out the interval in small intervals of $N^{-\delta}$. In other words, we have that $[0,T]=\underset{j \in J}{\bigcup}I_j$ where $I_j=[a_j, b_j]$, $|I_j|\thicksim N^{-\delta}$ and $\# J\sim N^{\delta}$.
Since $u_N$ is a solution to the integral equation
\begin{displaymath}
u_N(t) =e^{-(t-a_j)\partial_x^3}u_N(a_j)+\int_{a_j}^te^{-(t-t')\partial_x^3}F_N(t')dt'
\end{displaymath}
for $t \in I_j$, we deduce from \eqref{strichartz}-\eqref{strichartz2} that
\begin{displaymath}
\begin{split}
\|u_N\|_{L^4_TL^{\infty}_x} &\lesssim \Big(\sum_j \|D^{-\frac14}_xu_N(a_j)\|_{L^2_x}^4 \Big)^{\frac14}+
\Big(\sum_j \|D^{-\frac12}_xF_N \|_{L^\frac{4}{3}_{I_j} L^1_x}^4 \Big)^{\frac14} \\
& \lesssim N^{\frac{\delta}4}\|D^{-\frac14}_xu_N\|_{L^{\infty}_TL^2_x}
+\Big(\sum_j \Big(\int_{I_j}\|D^{-\frac12}_xF_N(t')\|_{L^1_x}^{\frac43} dt'\Big)^3 \Big)^{\frac14} \\
& \lesssim \|D^{\frac{\delta-1}4}_xu_N\|_{L^{\infty}_TL^2_x}+\Big(\sum_j |I_j|^2\int_{I_j}\|D^{-\frac12}_xF_N(t')\|_{L^1_x}^4dt' \Big)^{\frac14} \\
& \lesssim \|D^{\frac{\delta-1}4}_xu_N\|_{L^{\infty}_TL^2_x}+\|D^{-\frac{\delta+1}2}_xF_N\|_{L^4_TL^1_x} \, ,
\end{split}
\end{displaymath}
which concludes the proof of \eqref{refinedStrichartz2}.
\end{proof}

\section{$L^2$ multilinear estimates} \label{Secmultest}

In this section we follow some notations of \cite{Tao}. For $k \in \mathbb Z_+$ and $\xi \in \mathbb R$, let $\Gamma^k(\xi)$ denote the $k$-dimensional  \lq\lq affine hyperplane\rq\rq \, of $\mathbb R^{k+1}$ defined by
\begin{displaymath}
\Gamma^k(\xi)=\big\{ (\xi_1,\cdots,\xi_{k+1}) \in \mathbb R^{k+1} : \ \xi_1+\cdots + \xi_{k+1}=\xi\big\} \, ,
\end{displaymath}
and endowed with the obvious measure
\begin{displaymath}
\int_{\Gamma^k(\xi)}F =  \int_{\Gamma^k(\xi)}F(\xi_1,\cdots,\xi_{k+1}) :=
\int_{\mathbb R^k} F\big(\xi_1,\cdots,\xi_k,\xi-(\xi_1+\cdots+\xi_k)\big)d\xi_1\cdots d\xi_k \, ,
\end{displaymath}
for any function $F: \Gamma^k(\xi) \rightarrow \mathbb C$. When $\xi=0$, we simply denote $\Gamma^k=\Gamma^k(0)$ with the obvious modifications.

Moreover, given $T>0$, we also define $\mathbb R_T=\mathbb R \times [0,T]$ and $\Gamma^k_T=\Gamma^k \times [0,T]$ with the obvious measures
\begin{displaymath}
\int_{\mathbb R_T} u := \int_{\mathbb R \times [0,T]}u(x,t)dxdt \end{displaymath}
and
\begin{displaymath}
\int_{\Gamma^k_T} F :=\int_{\mathbb R^k\times [0,T]} F\big(\xi_1,\cdots,\xi_k,\xi-(\xi_1+\cdots+\xi_k),t\big)d\xi_1\cdots d\xi_k dt \, .
\end{displaymath}

\subsection{$L^2$ trilinear estimates}

\begin{lemma}\label{prod4-est}
 Let $f_j\in L^2(\mathbb R)$, $j=1,...,4$ and $M \in \mathbb D$. Then it holds that
\begin{equation} \label{prod4-est.1}
\int_{\Gamma^3} \phi_M(\xi_1+\xi_2) \prod_{j=1}^4|f_j(\xi_j)| \lesssim M \prod_{j=1}^4 \|f_j\|_{L^2} \, .
\end{equation}

\end{lemma}
\begin{proof}
Let us denote by $\mathcal{I}^3_M(f_1,f_2,f_3,f_4)$ the integral on the left-hand side of \eqref{prod4-est.1}. We can assume without loss of generality that $f_i \ge 0$ for $i=1,\cdots,4$. Then, we have that
\begin{equation}  \label{prod4-est.2}
\mathcal{I}^3_M(f_1,f_2,f_3,f_4) \le \mathcal{J}_M(f_1,f_2) \, \times\sup_{\xi_1,\xi_2} \int_{\mathbb R} f_3(\xi_3)f_4(-(\xi_1+\xi_2+\xi_3))d\xi_3 \, ,
\end{equation}
where
\begin{equation}  \label{prod4-est.3}
 \mathcal{J}_M(f_1,f_2)=\int_{\mathbb R^2}\phi_M(\xi_1+\xi_2)f_1(\xi_1)f_2(\xi_2)d\xi_1d\xi_2 \, .\end{equation}
H\"older's inequality yields
\begin{equation}  \label{prod4-est.4}
 \mathcal{J}_M(f_1,f_2)=\int_{\mathbb R}\phi_M(\xi_1)(f_1 \ast f_2) (\xi_1)d\xi_1 \lesssim M\|f_1 \ast f_2\|_{L^{\infty}} \lesssim M\|f_1\|_{L^2} \|f_2\|_{L^2} \, .
\end{equation}
Moreover, the Cauchy-Schwarz inequality yields
\begin{equation} \label{prod4-est.5}
\int_{\mathbb R} f_3(\xi_3)f_4(-(\xi_1+\xi_2+\xi_3))d\xi_3  \le \|f_3\|_{L^2} \|f_4\|_{L^2} \, .
\end{equation}
Therefore, estimate \eqref{prod4-est.1} follows from \eqref{prod4-est.2}--\eqref{prod4-est.5}.
\end{proof}

\medskip

For a fixed $N\ge 1$ dyadic, we introduce the following disjoint subsets of $\mathbb D^3$:
\begin{align*}
  \mathcal{M}_3^{low} &= \big\{(M_1,M_2,M_3)\in \D^3 \, : M_{min} \le N^{-\frac12} \textrm{ and } M_{med}\le 2^{-9}N\big\} \, ,\\
  \mathcal{M}_3^{med} &= \big\{(M_1,M_2,M_3)\in\D^3 \, : \, N^{-\frac12} < M_{min} \le M_{med} \le 2^{-9}N\big\} \, ,\\
  \mathcal{M}_3^{high} &= \big\{(M_1,M_2,M_3)\in\D^3 \, : \,   2^{-9}N < M_{med}  \big\} \, ,
\end{align*}
where $M_{min} \le M_{med} \le M_{max}$ denote respectively the minimum, sub-maximum and maximum of $\{ M_1,M_2,M_3\}$.

We will denote by $\phi_{M_1,M_2,M_3}$ the function
$$
\phi_{M_1,M_2,M_3}(\xi_1,\xi_2,\xi_3) = \phi_{M_1}(\xi_2+\xi_3) \phi_{M_2}(\xi_1+\xi_3) \phi_{M_3}(\xi_1+\xi_2).
$$

Next, we state a useful technical lemma.
\begin{lemma} \label{teclemma}
Let $(\xi_1,\xi_2,\xi_3) \in \mathbb R^3$ satisfy $|\xi_j| \sim N_j$ for $j=1,2,3$ and $|\xi_1+\xi_2+\xi_3| \sim N$. Let $(M_1,M_2,M_3) \in \mathcal{M}_3^{low} \cup \mathcal{M}_3^{med}$. Then it holds that
\begin{displaymath}
N_1 \sim N_2 \sim N_3\sim M_{max} \sim N \quad \text{if} \quad (\xi_1,\xi_2,\xi_3) \in \text{supp} \, \phi_{M_1,M_2,M_3} \, ,
\end{displaymath}	
\end{lemma}

\begin{proof} Without loss of generality, we can assume that $M_1 \le M_2 \le M_3$. Let $(\xi_1,\xi_2,\xi_3) \in \text{supp} \, \phi_{M_1,M_2,M_3}$. Then, we have $|\xi_2+\xi_3| \ll N$ and $|\xi_1+\xi_3| \ll N$, so that $N_1 \sim N_2 \sim N$ since $|\xi_1+\xi_2+\xi_3| \sim N$.
	
On one hand $N_3 \ll N$ would imply that $M_1 \sim M_2 \sim N$ which is a contradiction. On the other hand, $N_3 \gg N$ would imply that $|\xi_1+\xi_2+\xi_3| \gg N$ which is also a contradiction. Therefore, we must have $N_3 \sim N$.

Finally,  $M_1 \ll N$ implies that $\xi_2 \cdot \xi_3 <0$ and  $M_2 \ll N$ implies $\xi_1 \cdot \xi_3<0$. Thus, $\xi_1 \cdot \xi_2>0$, so that $M_3 \sim N$.	
\end{proof}

For $\eta\in L^\infty$, let us define the trilinear pseudo-product operator $\Pi^3_{\eta,M_1,M_2,M_3}$  in Fourier variables by
\begin{equation} \label{def.pseudoproduct3}
\mathcal{F}_x\big(\Pi^3_{\eta,M_1,M_2,M_3}(u_1,u_2,u_3) \big)(\xi)=\int_{\Gamma^2(\xi)}(\eta\phi_{M_1,M_2,M_3})(\xi_1,\xi_2,\xi_3)\prod_{j=1}^3\widehat{u}_j(\xi_j) \, .
\end{equation}
It is worth noticing that when the functions $u_j$ are real-valued, the Plancherel identity yields
\begin{equation} \label{def.pseudoproduct3b}
\int_{\mathbb R} \Pi^3_{\eta,M_1,M_2,M_3}(u_1,u_2,u_3) \, u_4 \, dx=\int_{\Gamma^3} \big(\eta \phi_{M_1,M_2,M_3}\big)(\xi_1,\xi_2,\xi_3) \prod_{j=1}^4 \widehat{u}_j(\xi_j) \, .
\end{equation}

Finally, we  define the resonance function of order $3$ by
\begin{align}
\Omega^3(\xi_1,\xi_2,\xi_3) &= \xi_1^3+\xi_2^3+\xi_3^3-(\xi_1+\xi_2+\xi_3)^3 \notag\\
&= -3(\xi_1+\xi_2)(\xi_1+\xi_3)(\xi_2+\xi_3) \, .\label{res3}
\end{align}

The following proposition gives suitable estimates for the pseudo-product $\Pi^3_{M_1,M_2,M_3}$ when $(M_1,M_2,M_3) \in \mathcal{M}_3^{high}$.
\begin{proposition} \label{L2trilin}
Let $N_i$, $i=1,\cdots,4,$ and $N$ denote nonhomogeneous dyadic numbers.
Assume that $0<T \le 1$, $\eta$ is a bounded function and $u_i$ are real-valued functions in $Y^0=X^{-1,1} \cap X^{-\frac{7}{8},\frac{15}{16}}\cap \widetilde{L^{\infty}_t}L^2_x$ with  spatial Fourier support in $I_{N_i}$ for $i=1,\cdots 4$.  Assume also that  $N \gg 1$, $(M_1,M_2,M_3)\in \mathcal{M}_3^{high}$ and  $ M_{min} \ge N^{-1} .$
Then
\begin{equation} \label{L2trilin.2}
\Big| \int_{\R \times [0,T]}\Pi^3_{\widetilde{\eta},M_1,M_2,M_3}(u_1,u_2,u_3) \, u_4 \, dxdt \Big| \lesssim N_{max}^{-1} (M_{min}\wedge 1)^{\frac{1}{16}}\prod_{i=1}^4\|u_i\|_{Y^0} \, .
\end{equation}
where $N_{max}=\max\{N_1,N_2,N_3\}$ and $\widetilde{\eta}=\eta \phi_N(\xi_1+\xi_2+\xi_3)$.

Moreover, the implicit constant in estimate \eqref{L2trilin.2} only depends on the $L^{\infty}$-norm of the function $\eta$.
 \end{proposition}

 Before giving the proof of Proposition \ref{L2trilin}, we state some important technical lemmas whose proofs can be found in \cite{MoVe}.
\begin{lemma} \label{QL}
 Let $L$ be a nonhomogeneous dyadic number. Then the operator $Q_{\le L}$ is bounded in $L^{\infty}_tL^2_x$ uniformly in $L$. In other words,
\begin{equation} \label{QL.1}
\|Q_{\le L}u\|_{L^{\infty}_tL^2_x} \lesssim \|u\|_{L^{\infty}_tL^2_x} \, ,
\end{equation}
for all $u \in L^{\infty}_tL^2_x$ and the implicit constant appearing in \eqref{QL.1} does not depend on $L$.
 \end{lemma}

\begin{proof}
See Lemma 2.3 in \cite{MoVe}.
\end{proof}

For any $0<T \le 1$, let us denote by $1_T$ the characteristic function of the interval $[0,T]$. One of the main difficulty in the proof of Proposition \ref{L2trilin} is that the operator of multiplication by $1_T$ does not commute with $Q_L$. To handle this situation, we follow the arguments introduced in \cite{MoVe} and use the decomposition
\begin{equation} \label{1T}
1_T=1_{T,R}^{low}+1_{T,R}^{high}, \quad \text{with} \quad \mathcal{F}_t\big(1_{T,R}^{low} \big)(\tau)=\chi(\tau/R)\mathcal{F}_t\big(1_{T} \big)(\tau) \, ,
\end{equation}
for some $R>0$ to be fixed later.
\begin{lemma}\label{ihigh-lem} For any $ R>0 $ and $ T>0 $ it holds
\begin{equation}\label{high}
\|1_{T,R}^{high}\|_{L^1}\lesssim T\wedge R^{-1} \, ,
\end{equation}
and
\begin{equation}\label{high2}
\| 1_{T,R}^{low}\|_{L^\infty}\lesssim 1 \; .
\end{equation}
\end{lemma}

\begin{proof}  See Lemma 2.4 in \cite{MoVe}.
\end{proof}

\begin{lemma}\label{ilow-lem}
Assume that $T>0$, $R>0$  and $ L \gg R $. Then, it holds
\begin{equation} \label{ihigh-lem.1}
\|Q_L (1_{T,R}^{low}u)\|_{L^2_{x,t}}\lesssim \|Q_{\sim L} u\|_{L^2_{x,t}} \, ,
\end{equation}
for all $u \in L^2(\mathbb R^2)$.
\end{lemma}

\begin{proof}
See Lemma 2.5 in \cite{MoVe}.
\end{proof}

\begin{proof}[Proof of Proposition \ref{L2trilin}] Given $u_i$, $1 \le i \le 4$, satisfying the hypotheses of Proposition \ref{L2trilin}, let $G_{M_1,M_2,M_3}^3=G_{M_1,M_2,M_3}^3(u_1,u_2,u_3,u_4)$ denote the left-hand side of \eqref{L2trilin.2}. We use the decomposition in \eqref{1T} and obtain that
\begin{equation} \label{L2trilin.4}
G_{M_1,M_2,M_3}^3=G_{M_1,M_2,M_3,R}^{3,low}+G_{M_1,M_2,M_3,R}^{3,high} \, ,
\end{equation}
where
\begin{displaymath}
G_{M_1,M_2,M_3,R}^{3,low}=\int_{\mathbb R^2}1^{low}_{T,R}\,\Pi_{\eta,M_1,M_2,M_3}^3(u_1,u_2,u_3) u_4 \, dxdt
\end{displaymath}
and
\begin{displaymath}
G_{M_1,M_2,M_3,R}^{3,high}=\int_{\mathbb R^2}1^{high}_{T,R}\,\Pi_{\eta,M_1,M_2,M_3}^3(u_1,u_2,u_3)u_4 \, dxdt \, .
\end{displaymath}

We deduce from H\"older's inequality in time, \eqref{prod4-est.1}, \eqref{def.pseudoproduct3b} and \eqref{high} that
\begin{displaymath}
\begin{split}
\big|G_{M_1,M_2,M_3,R}^{3,high}\big| & \le \|1_{T,R}^{high}\|_{L^1}\big\|\int_{\mathbb R}\Pi_{\eta,M_1,M_2,M_3}^3(u_1,u_2,u_3)u_4 \, dx\big\|_{L^{\infty}_t}
\\ & \lesssim  R^{-1}M_{min}\prod_{i=1}^4\|u_i\|_{L^{\infty}_tL^2_x} \, ,
\end{split}
\end{displaymath}
which implies that
\begin{equation} \label{L2trilin.5}
\big|G_{M_1,M_2,M_3,R}^{3,high}\big| \lesssim N_{max}^{-1}(M_{min}\wedge 1)^{\frac{1}{16}}\prod_{i=1}^4\|u_i\|_{L^{\infty}_tL^2_x}
\end{equation}
if we choose $R=M_{min}  (M_{min}\wedge 1)^{-\frac{1}{16}}N_{max}$.

To deal with the term $G_{M_1,M_2,M_3,R}^{3,low}$, we decompose with respect to the modulation variables. Thus,
\begin{displaymath}
G_{M,R}^{3,low}=\sum_{L_1,L_2,L_3,L_4}\int_{\mathbb R^2}\Pi_{\eta,M_1,M_2,M_3}^3(Q_{L_1}(1^{low}_{T,R}u_1),Q_{L_2}u_2,Q_{L_3}u_3)Q_{L_4}u_4 \, dxdt \, .
\end{displaymath}
Moreover, observe from the resonance relation in \eqref{res3} and the hypothesis $(M_1,M_2,M_3) \in \mathcal{M}_3^{high}$ that
\begin{equation} \label{L2trilin5b}
L_{max} \gtrsim M_{min}N_{max}^2 \, ,
\end{equation}
where $L_{max}=\max \{L_1,L_2,L_3,L_4 \}$.
In the case where $N_{max} \sim N$, \eqref{L2trilin5b} is clear from the definition of $\mathcal{M}^{high}$. In the case where $N_{max} \sim N_{med} \gg N$,  we claim  that $M_{max} \sim M_{med} \gtrsim N_{max}$. Indeed, denote $\{\xi_1,\xi_2,\xi_3 \}=\{\xi_{max},\xi_{med},\xi_{min} \}$, where $ |\xi_{min}| \le |\xi_{med}| \le |\xi_{max}|$. Then we compute, using also the hypothesis $|\xi_4| \sim N$,
$|\xi_{max}+\xi_{min}|=|\xi_4-\xi_{med}| \sim N_{med}$ and $|\xi_{med}+\xi_{min}| =|\xi_4-\xi_{max}| \sim N_{max}$, which proves the claim.

In particular, \eqref{L2trilin5b} implies that
$$L_{max} \gg R=M_{min}  (M_{min}\wedge 1)^{-\frac{1}{16}}N_{max},
$$
 since $N_{max} \gg 1$ and $ M_{min} \ge N^{-1} \gtrsim N_{max}^{-1} $.

In the case where $L_{max}=L_1$, we deduce from \eqref{prod4-est}, \eqref{def.pseudoproduct3b}, \eqref{QL.1}, \eqref{ihigh-lem.1} and \eqref{L2trilin5b} that
\begin{displaymath}
\begin{split}
\big| G_{M_1,M_2,M_3,R}^{3,low} \big| &\lesssim \sum_{L_1 \gtrsim M_{min}N_{max}^2}M_{min}\|Q_{L_1}(1^{low}_{T,R}u_1)\|_{L^2_{x,t}}\prod_{i=2}^4\|Q_{\le L_i}u_i\|_{L^{\infty}_tL^2_x} \\ & \lesssim N_{max}^{-1}(1 \wedge M_{min}^\frac{1}{16})(\|u_1\|_{X^{-1,1}}+\|u_1\|_{X^{-\frac{7}{8},\frac{15}{16}}})\prod_{i=2}^{4}\|u_i\|_{L^{\infty}_tL^2_x}  \, ,
\end{split}
\end{displaymath}
which implies that
\begin{equation} \label{L2trilin.6}
\big| G_{M_1,M_2,M_3,R}^{3,low} \big| \lesssim N_{max}^{-1}(1 \wedge M_{min}^\frac{1}{16})\prod_{i=1}^4\|u_i\|_{Y^0} \, .
\end{equation}

We can prove arguing similarly that \eqref{L2trilin.6} still holds true in all the other cases, \textit{i.e.} $L_{max}=L_2, \ L_3$ or $L$. Note that for those cases  we do not have to use \eqref{high} but we only need \eqref{high2}.
Therefore, we conclude the proof of estimate \eqref{L2trilin.2} gathering \eqref{L2trilin.4}, \eqref{L2trilin.5} and \eqref{L2trilin.6}.
\end{proof}

\subsection{$L^2$ 5-linear estimates}
\begin{lemma}\label{prod6-est}
 Let $f_j\in L^2(\R)$, $j=1,...,6$ and $M_1,M_4 \in \mathbb D$. Then it holds that
\begin{equation}\label{prod6.est.1}
\int_{\Gamma^5} \phi_{M_1}(\xi_2+\xi_3)\phi_{M_4}(\xi_5+\xi_6) \prod_{j=1}^6|f_j(\xi_j)| \lesssim M_1M_4 \prod_{j=1}^6 \|f_j\|_{L^2}.
\end{equation}
If moreover $f_j$ are localized in an annulus $\{|\xi|\sim N_j\}$ for $j=5,6$, then
\begin{equation} \label{prod6.est.2}
\int_{\Gamma^5} \phi_{M_1}(\xi_2+\xi_3)\phi_{M_4}(\xi_5+\xi_6) \prod_{i=1}^6|f_i(\xi_i)| \lesssim M_1M_4^{\frac12}N_5^{\frac14}N_6^{\frac14} \prod_{i=1}^6 \|f_i\|_{L^2} \, .
\end{equation}
\end{lemma}

\begin{proof}
Let us denote by $\mathcal{I}^5=\mathcal{I}^5(f_1,\cdots,f_6)$ the integral on the right-hand side of \eqref{prod6.est.1}. We can assume without loss of generality that $f_j \ge 0$, $j=1,\cdots,6$. We have by using the notation in \eqref{prod4-est.4} that
\begin{equation} \label{prod6.est.3}
\mathcal{I}^5\le \mathcal{J}_{M_1}(f_2,f_3) \times \mathcal{J}_{M_4}(f_5,f_6) \times
\sup_{\xi_2,\xi_3,\xi_5,\xi_6} \int_{\mathbb R} f_1(\xi_1)f_4(-\sum_{\genfrac{}{}{0pt}{}{j=1}{ j\neq 4}}^6\xi_j) \, d\xi_1 \, .
\end{equation}
Thus, estimate \eqref{prod6.est.1} follows applying \eqref{prod4-est.4} and the Cauchy-Schwarz inequality to \eqref{prod6.est.3}.

Assuming furthermore that $f_j$ are localized in an annulus $\{|\xi| \sim N_j\}$ for $j=5,6$, then we get arguing as above that
\begin{equation} \label{prod6.est.4}
\mathcal{I}^5\le  M_1\times\mathcal{J}_{M_4}(f_5,f_6) \times
 \prod_{j=1}^4\|f_j\|_{L^2}\, .
\end{equation}
From the Cauchy-Schwarz inequality
\begin{equation*}
\mathcal{J}_{M_4}(f_5,f_6) \le \Big( \int_{\mathbb R}f_5(\xi_5)d\xi_5\Big) \times \Big( \int_{\mathbb R} f_6(\xi_6)d\xi_6\Big) \lesssim N_5^{\frac12}N_6^{\frac12}\|f_5\|_{L^2}\|f_6\|_{L^2} \, ,
\end{equation*}
which together with \eqref{prod6.est.4} yields
\begin{equation} \label{prod6.est.6}
\mathcal{I}^5\lesssim M_1N_5^{\frac12}N_6^{\frac12} \prod_{i=1}^6 \|f_i\|_{L^2} \, .
\end{equation}
Therefore, we conclude the proof of \eqref{prod6.est.2} interpolating \eqref{prod6.est.1} and \eqref{prod6.est.6}.
\end{proof}

For a fixed $N\ge 1$ dyadic, we introduce the following subsets of $\D^6$:
\begin{align*}
  \mathcal{M}_5^{low} &= \big\{(M_1,...,M_6)\in\D^6 \, :  \, (M_1,M_2,M_3)\in \mathcal{M}_3^{med},  \ \\ &\quad  \quad \quad M_{min(5)}\le 2^9  M_{med(3)}
   \textrm{ and } \ M_{med(5)}\le 2^{-9}N \big\},\\
  \mathcal{M}_5^{med} &= \big\{(M_1,...,M_6)\in\D^6 \, : \, (M_1,M_2,M_3)\in \mathcal{M}_3^{med} \ \text{and}   \\
  & \quad \quad \quad 2^9 M_{med(3)} <M_{min(5)} \le M_{med(5)} \le 2^{-9}N
  \big\},\\
  \mathcal{M}_5^{high} &= \big\{(M_1,...,M_6)\in\D^6 \, : \, (M_1,M_2,M_3)\in \mathcal{M}_3^{med} \ \text{and} \  2^{-9}N < M_{med(5)}\big\} \, ,
\end{align*}
where $M_{max(3)} \ge M_{med(3)} \ge M_{min(3)}$, respectively $M_{max(5)} \ge M_{med(5)} \ge M_{min(5)}$, denote the maximum, sub-maximum and minimum of $\{M_1, M_2, M_3\}$, respectively $\{M_4,M_5,M_6\}$.
We will also denote by $\phi_{M_1,...,M_6}$ the function defined on $\R^6$ by
\begin{equation*}
\phi_{M_1,...,M_6}(\xi_1,...,\xi_6) = \phi_{M_1,M_2,M_3}(\xi_1,\xi_2,\xi_3) \phi_{M_4,M_5,M_6}(\xi_4,\xi_5,\xi_6)\, .
\end{equation*}
For $\eta\in L^\infty$, let us define the operator $\Pi^5_{\eta,M_1,...,M_6}$ in Fourier variables by
\begin{equation} \label{def.pseudoproduct6}
\mathcal{F}_x\big(\Pi^5_{\eta,M_1,...,M_6}(u_1,...,u_5) \big)(\xi)=\int_{\Gamma^4(\xi)}(\eta\phi_{M_1,...,M_6})(\xi_1,...,\xi_5,-\sum_{j=1}^5\xi_j) \prod_{j=1}^5 \widehat{u}_j(\xi_j) \, .
\end{equation}
Observe that, if the functions $u_j$ are real valued, the Plancherel identity yields
\begin{equation} \label{def.pseudoproduct6b}
\int_{\mathbb R}\Pi^5_{\eta,M_1,...,M_6}(u_1,...,u_5)\, u_6 \, dx=\int_{\Gamma^5}(\eta\phi_{M_1,...,M_6})\prod_{j=1}^6 \widehat{u}_j(\xi_j) \, .
\end{equation}

Finally, we  define the resonance function of order $5$ for $\vec{\xi}_{(5)}=(\xi_1,\cdots,\xi_6) \in \Gamma^5$ by
\begin{equation} \label{res5}
\Omega^5(\vec{\xi}_{(5)}) = \xi_1^3+\xi_2^3+\xi_3^3+\xi_4^3+\xi_5^3+\xi_6^3 \; .
\end{equation}
It is worth noticing that a direct calculus leads to
\begin{equation}\label{res55}
\Omega^5(\vec{\xi}_{(5)}) = \Omega^3(\xi_1,\xi_2,\xi_3) + \Omega^3(\xi_4,\xi_5,\xi_6) \, .
\end{equation}

\medskip

The following proposition gives suitable estimates for the pseudo-product $\Pi^5_{M_1,\cdots,M_6}$ when $(M_1,\cdots,M_6) \in \mathcal{M}^{high}_5$ in the non resonant case $M_1M_2M_3\not\sim M_4M_5M_6$.
\begin{proposition} \label{L25lin}
Let $N_i$, $i=1,\cdots,6$ and $N$ denote nonhomogeneous dyadic numbers.
Assume that $0<T \le 1$, $\eta$ is a bounded function and $u_i$ are functions in $X^{-1,1} \cap L^{\infty}_tL^2_x$ with  spatial Fourier support in $I_{N_i}$ for $i=1,\cdots 6$.
If $N \gg 1$ and $(M_1,...,M_6)\in \mathcal{M}_5^{high}$ satisfies the non resonance assumption $M_1M_2M_3\not\sim M_4M_5M_6$, then
\begin{equation} \label{L25lin.2}
\big| \int_{\R \times [0,T]}\Pi^5_{\widetilde{\eta},M_1,...,M_6}(u_1,...,u_5)\, u_6 \, dxdt\big| \lesssim M_{min(3)}N_{max(5)}^{-1} \prod_{i=1}^6(\|u_i\|_{X^{-1,1}}+\|u_i\|_{L^\infty_t L^2_x}) \, ,
\end{equation}
where $N_{max(5)}=\max\{N_4,N_5,N_6\}$ and $\widetilde{\eta}=\eta \phi_N(\xi_1+\xi_2+\xi_3)$.

Moreover, the implicit constant in estimate \eqref{L25lin.2} only depends on the $L^{\infty}$-norm of the function $\eta$.
\end{proposition}

 \begin{proof}
The proof is similar to the proof of Proposition \ref{L2trilin}. We may always assume $M_1\le M_2\le M_3$ and $M_4\le M_5\le M_6$.

Since $|\xi_1+\xi_2+\xi_3|=|\xi_4+\xi_5+\xi_6| \sim N$ and $(M_1,\cdots,M_6) \in \mathcal{M}_5^{high}$, we get from Lemma \ref{teclemma} that $N_1 \sim N_2 \sim N_3 \sim N$, so that $N_{max(5)} \sim \max\{N_1,\cdots,N_6\}$. Moreover, it follows arguing as in the proof of \eqref{L2trilin5b} that $M_4M_5M_6 \gtrsim M_4N_{max(5)}^2$.
Hence, we deduce from identities \eqref{res55} and \eqref{res3} and the non resonance assumption  that
\begin{equation} \label{L25lin.20}
L_{max} \gtrsim \max(M_1M_2M_3, M_4M_5M_6) \gtrsim M_4M_5M_6 \gtrsim M_{4}N_{max(5)}^2 \, .
\end{equation}
Estimate \eqref{L25lin.2} follows then from estimates \eqref{L25lin.20} and \eqref{prod6.est.1} arguing as in the proof of Proposition \ref{L2trilin}.
\end{proof}

\subsection{$L^2$ 7-linear estimates}
 \begin{lemma}\label{prod8-est}
  Let $f_i\in L^2(\R)$, $i=1,...,8$ and $M_1,  \, M_4, \, M_6$ and $M_7 \in \mathbb D$. Then it holds that
\begin{equation}\label{prod8-est.1}
\int_{\Gamma^7} \phi_{M_1}(\xi_2+\xi_3) \phi_{M_6}(\xi_4+\xi_5)\phi_{M_7}(\xi_7+\xi_8) \prod_{i=1}^8|f_i(\xi_i)| \lesssim M_1M_6M_7 \prod_{i=1}^8 \|f_i\|_{L^2} \,
\end{equation}
and
\begin{equation} \label{prod8-est.100}
\int_{\Gamma^7} \phi_{M_1}(\xi_2+\xi_3) \phi_{M_4}(\sum_{j=1}^4\xi_j)\phi_{M_7}(\xi_7+\xi_8) \prod_{i=1}^8|f_i(\xi_i)| \lesssim M_1M_4M_7 \prod_{i=1}^8 \|f_i\|_{L^2} \, .
\end{equation}

If moreover $f_j$ is localized in an annulus $\{|\xi|\sim N_j\}$ for $j=7, \, 8$, then
\begin{equation} \label{prod8-est.0b}
\begin{split}
 \int_{\Gamma^7} \phi_{M_1}(\xi_2+\xi_3)\phi_{M_6}(\xi_4+\xi_5) &\phi_{M_7}(\xi_7+\xi_8) \prod_{i=1}^8|f_i(\xi_i)| \\ & \lesssim M_1M_6M_7^{\frac12}N_7^{\frac14}N_8^{\frac14} \prod_{i=1}^8 \|f_i\|_{L^2} \, .
\end{split}
\end{equation}
and
\begin{equation} \label{prod8-est.200}
\begin{split}
\int_{\Gamma^7} \phi_{M_1}(\xi_2+\xi_3) \phi_{M_4}(\sum_{j=1}^4\xi_j)&\phi_{M_7}(\xi_7+\xi_8) \prod_{i=1}^8|f_i(\xi_i)| \\ &\lesssim M_1M_4M_7^{\frac12} N_7^{\frac14}N_8^{\frac14} \prod_{i=1}^8 \|f_i\|_{L^2} \, .
\end{split}
\end{equation}
\end{lemma}

\begin{proof}
 Let us denote by $\mathcal{I}^7=\mathcal{I}^7(f_1,\cdots,f_8)$ the integral on the right-hand side of \eqref{prod8-est.1}. We can assume without loss of generality that $f_j \ge 0$, $j=1,\cdots,8$. We have by using the notation in \eqref{prod4-est.4} that
\begin{equation} \label{prod8-est.3}
\mathcal{I}^7\le \mathcal{J}_{M_1}(f_2,f_3) \times \mathcal{J}_{M_6}(f_4,f_5) \times \mathcal{J}_{M_7}(f_7,f_8) \times
\sup_{\xi_2,\xi_3,\xi_4,\xi_5,\xi_7,\xi_8} \int_{\mathbb R} f_1(\xi_1)f_6(-\sum_{\genfrac{}{}{0pt}{}{j=1}{ j\neq 6}}^8\xi_j) \, d\xi_1 \, .
\end{equation}
Thus, estimate \eqref{prod8-est.1} follows applying \eqref{prod4-est.4} and the Cauchy-Schwarz inequality to \eqref{prod8-est.3}.

Assuming furthermore that $f_j$ are localized in an annulus $\{|\xi| \sim N_j\}$ for $j=7,8$, then we get arguing as above that
\begin{equation} \label{prod8-est.4}
\mathcal{I}^7\le  M_1M_6 \times\mathcal{J}_{M_7}(f_7,f_8) \times
 \prod_{j=1}^6\|f_j\|_{L^2}\, .
\end{equation}
From the Cauchy-Schwarz inequality
\begin{equation} \label{prod8-est.5}
\mathcal{J}_{M_7}(f_7,f_8) \le \Big( \int_{\mathbb R}f_7(\xi_7)d\xi_7\Big) \times \Big( \int_{\mathbb R} f_8(\xi_8)d\xi_8\Big) \lesssim N_7^{\frac12}N_8^{\frac12}\|f_7\|_{L^2}\|f_8\|_{L^2} \, ,
\end{equation}
which together with \eqref{prod8-est.4} yields
\begin{equation} \label{prod8-est.6}
\mathcal{I}^7\lesssim M_1M_6N_7^{\frac12}N_8^{\frac12} \prod_{j=1}^8 \|f_j\|_{L^2} \, .
\end{equation}
Therefore, we conclude the proof of \eqref{prod8-est.0b} interpolating \eqref{prod8-est.1} and \eqref{prod8-est.6}.

Now, we prove estimate  \eqref{prod8-est.100}. Let us denote by $\widetilde{\mathcal{I}}^7=\widetilde{\mathcal{I}}^7(f_1,\cdots,f_8)$ the integral on the right-hand side of \eqref{prod8-est.100}. We can assume without loss of generality that $f_j \ge 0$, $j=1,\cdots,8$.

Let us define
\begin{displaymath}
 \mathcal{J}_M(f_1,f_2)(\xi)=\int_{\mathbb R^2}\phi_M(\xi_1+\xi_2+\xi)f_1(\xi_1)f_2(\xi_2)d\xi_1d\xi_2 \, .
\end{displaymath}
Hence, we have, by using the notation in \eqref{prod4-est.3}, $ \mathcal{J}_M(f_1,f_2)(0)= \mathcal{J}_M(f_1,f_2)$. Moreover, it follows from Young's inequality on convolution that
\begin{equation} \label{prod8-est.101}
\sup_{\xi}  \mathcal{J}_M(f_1,f_2)(\xi) = \sup_{\xi} \int_{\mathbb R} \phi_{M}(\xi_1) f_1 \ast f_2(\xi_1-\xi) \, d\xi_1
\lesssim M \|f_1 \ast f_2\|_{L^{\infty}} \lesssim M \|f_1\|_{L^2} \|f_2\|_{L^2}
\end{equation}

By using this notation and the fact that $\sum_{j=1}^4\xi_j=-\sum_{j=5}^8\xi_j$, we have
\begin{equation} \label{prod8-est.102}
\begin{split}
\widetilde{\mathcal{I}}^7\le \mathcal{J}_{M_1}(f_2,f_3) &\times \sup_{\xi_7,\xi_8}\mathcal{J}_{M_4}(f_5,f_6)(\xi_7+\xi_8) \times \mathcal{J}_{M_7}(f_7,f_8)\\ & \times
\sup_{\xi_2,\xi_3,\xi_5,\xi_6,\xi_7,\xi_8} \int_{\mathbb R} f_1(\xi_1)f_4(-\sum_{\genfrac{}{}{0pt}{}{j=1}{ j\neq 4}}^8\xi_j) \, d\xi_1 \, .
\end{split}
\end{equation}
Hence, we conclude the proof of \eqref{prod8-est.100} by applying \eqref{prod4-est.4}, \eqref{prod8-est.101} and the Cauchy-Schwarz inequality to \eqref{prod8-est.102}.

The proof of \eqref{prod8-est.200} follows arguing as above and using \eqref{prod8-est.5} to estimate $\mathcal{J}_{M_7}(f_7,f_8)$.
\end{proof}

\smallskip
For a fixed $N\ge 1$ dyadic, we introduce the following subsets of $\D^9$:
\begin{align*}
  \mathcal{M}_7^{low} &= \big\{(M_1,...,M_9)\in\D^9 \, :  \, (M_1,\cdots,M_6)\in \mathcal{M}_5^{med},  \\ &\quad  \quad \quad \quad \quad \quad \quad \quad M_{min(7)}\le 2^9  M_{med(5)}
   \textrm{ and } \ M_{med(7)}\le 2^{-9}N \big\},\\
  \mathcal{M}_7^{med} &= \big\{(M_1,...,M_9)\in\D^9 \, : \, (M_1,...,M_6)\in \mathcal{M}_5^{med},  \\ &\quad  \quad \quad \quad \quad \quad \quad \quad  2^9 M_{med(5)} <M_{min(7)} \le M_{med(7)} \le 2^{-9}N \big\} \, ,\\
  \mathcal{M}_7^{high} &= \big\{(M_1,...,M_9)\in\D^9 \, : \, (M_1,...,M_6)\in \mathcal{M}_5^{med},  \ 2^{-9}N < M_{med(7)}\big\}
  \end{align*}
where $M_{max(7)} \ge M_{med(7)} \ge M_{min(7)}$ denote respectively the maximum, sub-maximum and minimum of $\{M_7,M_8,M_9\}$.

We will denote by $\phi_{M_1,...,M_9}$ the function defined on $\Gamma_7$ by
\begin{equation} \label{def.pseudoproduct80}
\phi_{M_1,...,M_9}(\xi_1,...,\xi_7,\xi_8) = \phi_{M_1,...,M_6}(\xi_1,...,\xi_5,-\sum_{j=1}^5\xi_j) \, \phi_{M_7,M_8,M_9}(\xi_6,\xi_7,\xi_8) \, .
\end{equation}
For $\eta\in L^\infty$, let us define the operator $\Pi^7_{\eta,M_1,...,M_9}$ in Fourier variables by
\begin{equation} \label{def.pseudoproduct8}
\mathcal{F}_x\big(\Pi^7_{\eta,M_1,...,M_9}(u_1,...,u_7) \big)(\xi)=\int_{\Gamma^6(\xi)}(\eta\phi_{M_1,...,M_9})(\xi_1,...,\xi_7,-\xi) \prod_{j=1}^7 \widehat{u}_j(\xi_j) \, .
\end{equation}
Observe that, if the functions $u_j$ are real valued, the Plancherel identity yields
\begin{equation} \label{def.pseudoproduct8b}
\int_{\mathbb R}\Pi^7_{\eta,M_1,...,M_9}(u_1,...,u_7)\, u_8 \, dx=\int_{\Gamma^7}(\eta\phi_{M_1,...,M_9})\prod_{j=1}^8 \widehat{u}_j(\xi_j) \, .
\end{equation}

We define the resonance function of order $7$ for $\vec{\xi}_{(7)}=(\xi_1,\cdots,\xi_8) \in \Gamma^7$ by
\begin{equation} \label{res7}
\Omega^7(\vec{\xi}_{(7)}) = \sum_{j=1}^8\xi_j^3 \, .
\end{equation}
Again it is direct to check that
\begin{equation} \label{res77}
\Omega^7(\vec{\xi}_{(7)})
=\Omega^5(\xi_1,...,\xi_5, -\sum_{i=1}^5 \xi_i) + \Omega^3(\xi_6,\xi_7,\xi_8) \, .
\end{equation}

The following proposition gives suitable estimates for the pseudo-product $\Pi^7_{M_1,\cdots,M_9}$ when $(M_1,\cdots,M_9) \in \mathcal{M}^{high}_7$ in the nonresonant case $M_4M_5M_6\not\sim M_7M_8M_9$.
\begin{proposition} \label{L27lin}
Let $N_i$, $i=1,\cdots,8$ and $N$ denote nonhomogeneous dyadic numbers.
Assume that $0<T \le 1$, $\eta$ is a bounded function and $u_j$ are functions in $X^{-1,1} \cap L^{\infty}_tL^2_x$ with spatial Fourier support in $I_{N_j}$ for $j=1,\cdots 8$.

(a) Assume that $N \gg 1$ and $(M_1,...,M_9)\in \mathcal{M}_7^{high}$ satisfies the non resonance assumption $M_4M_5M_6\not\sim M_7M_8M_9$.
Then
\begin{equation} \label{L27lin.2}
\begin{split}
\Big| \int_{\R \times [0,T]}\Pi^7_{\widetilde{\eta},M_1,...,M_9}&(u_1,...,u_7) \, u_8 \, dxdt\Big|  \\ &
\lesssim M_{min(3)}M_{min(5)}N_{max(7)}^{-1} \prod_{j=1}^8(\|u_j\|_{X^{-1,1}}+\|u_j\|_{L^\infty_t L^2_x})  \, ,
\end{split}
\end{equation}
where $N_{max(7)}=\max\{N_6,N_7,N_8\}$ and $\widetilde{\eta}=\eta \phi_N(\xi_1+\xi_2+\xi_3)$.
\smallskip

(b) Assume that $N \gg 1$ and  $(M_1,...,M_9)\in \mathcal{M}_7^{med}$.
Then
\begin{equation} \label{L27lin.200}
\begin{split}
\Big| \int_{\R \times [0,T]}\Pi^7_{\widetilde{\eta},M_1,...,M_9}&(u_1,...,u_7) \, u_8 \, dxdt\Big|  \\ &
\lesssim \frac{M_{min(3)}M_{min(5)}}{M_{med(7)}} \prod_{j=1}^8(\|u_j\|_{X^{-1,1}}+\|u_j\|_{L^\infty_t L^2_x})  \, ,
\end{split}
\end{equation}
where $\widetilde{\eta}=\eta \phi_N(\xi_1+\xi_2+\xi_3)$.

Moreover, the implicit constants in estimates \eqref{L27lin.2} and \eqref{L27lin.200} only depend on the $L^{\infty}$-norm of the function $\eta$.
 \end{proposition}

 \begin{proof} The proof is similar to the proof of Proposition \ref{L2trilin}.

Under the assumptions in (a) $|\xi_1+\xi_2+\xi_3| \sim N$ and $(M_1,\cdots,M_9) \in \mathcal{M}_7^{high}$, we get by using twice Lemma \ref{teclemma} that $|\xi_1| \sim |\xi_2| \sim |\xi_3| \sim |\xi_4| \sim |\xi_5| \sim |\xi_6+\xi_7+\xi_8| \sim N$, so that $N_{max(7)} \sim \max\{N_1,\cdots,N_8\}$.
On the one hand, since $(M_1,\cdots,M_6) \in \mathcal{M}_5^{med}$, it is clear that $M_4M_5M_6 \gg M_1M_2M_3$. On the other hand, it follows arguing as in the proof of \eqref{L2trilin5b} that $M_7M_8M_9 \gtrsim M_{min(7)}N_{max(7)}^2$.
Hence, we deduce from identities \eqref{res55},  \eqref{res77} and the non resonance assumption that
\begin{displaymath}
L_{max} \gtrsim \max(M_4 M_5 M_6, M_7 M_8 M_9) \gtrsim M_{min(7)} \, N_{max(7)}^2 \, .
\end{displaymath}

Under the assumptions in (b) $|\xi_1+\xi_2+\xi_3| \sim N$ and $(M_1,\cdots,M_9) \in \mathcal{M}_7^{med}$, we get that $M_7M_8M_9 \gg M_4M_5M_6 \gg M_1M_2M_3$. We also have by applying three times Lemma \ref{teclemma} that $N_1 \sim \cdots \sim N_8 \sim M_{max(7)} \sim N$. Hence, we deduce that
\begin{displaymath}
L_{max} \gtrsim M_{min(7)} M_{med(7)} N \, .
\end{displaymath}

Estimates \eqref{L27lin.2} and \eqref{L27lin.200}  follow from these claims and estimates \eqref{prod8-est.1} and \eqref{prod8-est.100}, arguing as in the proof of Proposition \ref{L2trilin}.

Indeed, in view of the definition of $\phi_{M_1,\cdots,M_9}$ in \eqref{def.pseudoproduct80}, we can always assume by symmetry that $M_{min(3)}=M_1$ and $M_{min(7)}=\min(M_7,M_8,M_9)=M_7$. In the case where $M_{min(5)}=M_6$,  we use \eqref{prod8-est.1}, whereas in the case where $M_{min(5)}=M_4$, we  use  \eqref{prod8-est.100}. By symmetry, the case where $M_{min(5)}=M_5$ is equivalent to the case where  $M_{min(5)}=M_4$.
 \end{proof}

\section{Energy estimates} \label{Secenergy}

The aim of this section is to derive energy estimates for the solutions of \eqref{mKdV} and the solutions of the equation satisfied by the difference of two solutions of \eqref{mKdV}  (see equation \eqref{diffmKdV} below).

In order to simplify the notations in the proofs below, we will instead derive energy estimates on the solutions $u$ of the more general equation
\begin{equation}\label{eq-u0}
 \partial_t u+\partial_x^3 u = c_4\partial_x(u_1u_2u_3) \, ,
\end{equation}
where for any $i\in\{1,2,3\}$, $u_i$ solves
\begin{equation}\label{eq-u1}
  \partial_t u_i+\partial_x^3 u_i = c_i\partial_x(u_{i,1}u_{i,2}u_{i,3}) \, .
\end{equation}
Finally we also assume that each $u_{i,j}$ solves
\begin{equation}\label{eq-u2}
  \partial_t u_{i,j}+\partial_x^3 u_{i,j} = c_{i,j}\partial_x(u_{i,j,1}u_{i,j,2}u_{i,j,3}) \, ,
\end{equation}
for any $(i,j) \in \{1,2,3\}^2$.
We will sometimes use $u_4, \, u_{4,1}, \,  u_{4,2}, \,  u_{4,3}$ to denote respectively $u, \, u_1,\, u_2,\, u_3$. Here $c_j, \, j \in \{1,\cdots,4\}$ and $c_{i,j}, \, (i,j) \in \{1,2,3\}^2$ denote real constants. Moreover, we assume that all the functions appearing in \eqref{eq-u0}-\eqref{eq-u1}-\eqref{eq-u2} are real-valued.

\medskip

Also, we will use the notations defined at the beginning of Section \ref{Secmultest}.

\medskip

The main obstruction to estimate $\frac{d}{dt} \| P_Nu \|_{L^2}^2$ at this level of regularity is the resonant term $\int \partial_x\big(P_{+N}u_1P_{+N}u_2 P_{-N}u_3 \big) \, P_{-N}u \, dx$ for which the resonance relation \eqref{res3} is not strong enough. In this section we modify the energy by a fourth order term, whose part of the time derivative coming from the linear contribution of \eqref{eq-u0} will  cancel out this resonant term. Note that we also need to add a second modification to the energy to control the part of the time derivative of the first modification coming from the resonant nonlinear contribution of \eqref{eq-u0}.

\subsection{Definition of the modified energy}

Let $N_0 =2^9$ and $N$ be a nonhomogeneous dyadic number. For $t \ge 0$, we define the modified energy at the dyadic frequency $N$ by
\begin{equation} \label{defEN}
  \mathcal{E}_N(t) = \left\{ \begin{array}{ll} \frac 12 \|P_Nu(\cdot,t)\|_{L^2_x}^2 & \text{for} \ N \le N_0 \, \\
\frac 12 \|P_Nu(\cdot,t)\|_{L^2_x}^2 + \alpha \mathcal{E}_N^{3}(t)  +\beta \mathcal{E}_N^5(t) & \text{for} \ N> N_0 \, , \end{array}\right.
\end{equation}
where $\alpha$ and $\beta$ are real constants to be determined later,
\begin{displaymath}
\mathcal{E}_N^{3}(t) = \sum_{(M_1,M_2,M_3)\in\mathcal{M}_3^{med}} \int_{\Gamma^3}\phi_{M_1,M_2,M_3}\big(\vec{\xi}_{(3)}\big)  \phi_N^2(\xi_4) \frac{\xi_4}{\Omega^3(\vec{\xi}_{(3)})} \prod_{j=1}^4 \widehat{u}_j(t,\xi_j) \, ,
\end{displaymath}
where $\vec{\xi}_{(3)}=(\xi_1,\xi_2,\xi_3)$,
and
\begin{equation*}
\begin{split}
\mathcal{E}_N^5(t) &= \sum_{(M_1,...,M_6)\in \mathcal{M}_5^{med}} \sum_{j=1}^4 c_j \int_{\Gamma^5} \phi_{M_1,...,M_6}(\vec{\xi_j}_{(5)}) \phi_N^2(\xi_4) \frac{\xi_4 \xi_j}{\Omega^3(\vec{\xi_j}_{(3)})\Omega^5(\vec{\xi_j}_{(5)})}
\\ & \quad \quad \quad \quad \quad \quad \quad \quad \quad \quad \quad \quad
\times\prod_{\genfrac{}{}{0pt}{}{k=1}{ k\neq j}}^4 \widehat{u}_k(t,\xi_k) \prod_{l=1}^3 \widehat{u}_{j,l}(t,\xi_{j,l}) \, ,
\end{split}
\end{equation*}
with the convention $\displaystyle{\xi_j=-\sum_{\genfrac{}{}{0pt}{}{k=1}{k \neq j}}^4\xi_k=\sum_{l=1}^3\xi_{j,l}}$ and the notations $$\vec{\xi_j}_{(5)}=(\vec{\xi_j}_{(3)},\xi_{j,1},\xi_{j,2},\xi_{j,3}) \in \Gamma^5$$
with
\begin{displaymath}
\vec{\xi_1}_{(3)}=(\xi_2,\xi_3,\xi_4), \ \vec{\xi_2}_{(3)}=(\xi_1,\xi_3,\xi_4), \ \vec{\xi_3}_{(3)}=(\xi_1,\xi_2,\xi_4), \ \vec{\xi_4}_{(3)}=(\xi_1,\xi_2,\xi_3) \, .
\end{displaymath}

For $T >0$, we define the modified energy by using a nonhomogeneous dyadic decomposition in spatial frequency
\begin{equation}\label{def-EsT}
E^s_T(u) =  \sum_{N} N^{2s} \sup_{t \in [0,T]} \big|\mathcal{E}_N(t)\big| \, .
\end{equation}
By convention, we also set $E^s_0(u) =  \displaystyle{\sum_{N} N^{2s}  \big|\mathcal{E}_N(0)\big|}$.

The next lemma ensures that, for $s \ge\frac14$, the energy $E^s_T(u)$ is coercive in a small ball of $ H^s$ centered at the origin.
\begin{lemma}[Coercivity of the modified energy]\label{lem-EsT}
Let $s \ge 1/4$ and $u, u_i, u_{i,j}\in H^s_x$. Then it holds
\begin{equation} \label{lem-Est.1}
\|u\|_{\widetilde{L^\infty_T}H^s_x}^2 \lesssim E_T^s(u) +   \prod_{j=1}^4\|u_j\|_{\widetilde{L^\infty_T} H^s_x} + \sum_{j=1}^4 \prod_{k=1\atop k\neq j}^4 \|u_k\|_{\widetilde{L^\infty_T} H^s_x} \prod_{l=1}^3 \|u_{j,l}\|_{\widetilde{L^\infty_T} H^s_x} \, .
\end{equation}
\end{lemma}

\begin{proof}
We infer from (\ref{def-EsT}) and the triangle inequality that
\begin{equation} \label{lem-Est.2}
\|u\|_{\widetilde{L^\infty_T}H^s_x}^2 \lesssim E_T^s(u) + \!\sum_{N\ge N_0}N^{2s} \sup_{t \in [0,T]}\big|\mathcal{E}_N^{3}(t)\big|+\! \sum_{N\ge N_0}N^{2s}\sup_{t \in [0,T]}\big|\mathcal{E}_N^5(t)\big| \, .
\end{equation}

We first estimate the contribution of $\mathcal{E}_N^{3}$. By symmetry, we can always assume that $M_1 \le M_2 \le M_3$, so that we have $N^{-\frac12} < M_1 \le M_2 \ll N$ and $M_3 \sim N$, since $(M_1,M_2,M_3) \in \mathcal{M}_3^{med}$. Then, we have from Lemma \ref{prod4-est},
\begin{equation} \label{lem-Est.3}
\begin{split}
N^{2s}\big|\mathcal{E}_N^{3}(t)\big| &\lesssim \sum_{N^{-\frac12}<M_1,M_2\ll N\atop M_3\sim N} \frac{N^{2s+1}}{M_1M_2M_3} M_1 \prod_{j=1}^4\|P_{\sim N}u_{j}(t)\|_{L^2_x} \\ &
\lesssim N^{\frac12-2s} \prod_{j=1}^4\|P_{\sim N}u_{j}(t)\|_{H^s_x} \, ,
\end{split}
\end{equation}
where we used that $\displaystyle{\sum_{N^{-\frac12}<M_1,M_2\ll N}\frac1{M_2} \lesssim \sum_{N^{-\frac12}<M_1\ll N}\frac1{M_1} \lesssim N^{-\frac12}}$.

To estimate the contribution of $\mathcal{E}_N^5(t)$, we notice from Lemma \ref{teclemma} that for $(M_1,...,M_6)\in \mathcal{M}_5^{med}$, the integrand in the definition of $\mathcal{E}_N^5$ vanishes unless $|\xi_1|\sim ...\sim |\xi_4| \sim N$ and $|\xi_{j,1}| \sim |\xi_{j,2}| \sim |\xi_{j,3}| \sim N$. Moreover, we assume without loss of generality $M_1\le M_2\le M_3$ and $M_4\le M_5\le M_6$, so that
\begin{displaymath}
\left|\frac{\xi_4 \xi_j}{\Omega^3(\vec{\xi_j}_{(3)})\Omega^5(\vec{\xi_j}_{(6)})} \right| \sim \frac{N^2}{M_1M_2N\cdot M_4M_5N}\sim \frac{1}{M_1M_2M_4M_5} \, .
\end{displaymath}
Thus we infer from \eqref{prod6.est.1} that
\begin{align} \label{lem-Est.4}
N^{2s}|\mathcal{E}_N^5(t)| &\lesssim \sum_{j=1}^4 \sum_{N^{-\frac12} \le M_1\le M_2 \atop  M_2\le M_4\le M_5 \ll N} \frac{N^{2s}}{M_2M_5} \prod_{k=1\atop k\neq j}^4 \|P_{\sim N}u_k(t)\|_{L^2_x} \prod_{l=1}^3 \|P_{\sim N}u_{j,l}(t)\|_{L^2_x} \nonumber \\ & \lesssim
 N^{1-4s} \sum_{j=1}^4 \prod_{k=1\atop k\neq j}^4 \|P_{\sim N}u_k(t)\|_{H^s_x} \prod_{l=1}^3 \|P_{\sim N}u_{j,l}(t)\|_{H^s_x} \, .
\end{align}

Finally, we conclude the proof of \eqref{lem-Est.1} by summing  \eqref{lem-Est.3} and \eqref{lem-Est.4}  over the dyadic  $ N\ge N_0 $, with $ s>1/4 $, and using
 \eqref{lem-Est.2}.
\end{proof}

\begin{remark} Arguing as in the proof of Lemma \ref{lem-EsT}, we get that
\begin{equation} \label{lem-Est.100}
E_0^s(u) \lesssim \|u(0)\|_{H^s}^2+   \prod_{j=1}^4\|u_j(0)\|_{H^s_x} + \sum_{j=1}^4 \prod_{k=1\atop k\neq j}^4 \|u_k(0)\|_{H^s_x} \prod_{l=1}^3 \|u_{j,l}(0)\|_{H^s_x}
\end{equation}
as soon as $s \ge 1/4$.
\end{remark}

\subsection{Estimates for the modified energy}

\begin{proposition}\label{prop-ee}
Let $s>\frac13$, $0<T \le 1$ and $u, \, u_i , \,u_{i,j}\in Y^s_T$ be solutions of (\ref{eq-u0})-(\ref{eq-u1})-(\ref{eq-u2}) on $]0,T[$. Then we have
\begin{equation} \label{prop-ee.1}
\begin{split}
 E^s_T(u) &\lesssim E^s_0(u) + \prod_{j=1}^4\|u_j\|_{Y^s_T} + \sum_{j=1}^4 \prod_{k=1\atop k\neq j}^4 \|u_k\|_{Y^s_T} \prod_{l=1}^3 \|u_{j,l}\|_{Y^s_T} \\ & \quad
+\sum_{j=1}^4\sum_{m=1 \atop m \neq j}^4\prod_{k=1 \atop k \neq j}^4\|u_k\|_{Y^s_T} \prod_{l=1 \atop l\neq m}^3 \|u_{j,l}\|_{Y^s_T} \prod_{n=1}^3\|u_{j,m,n}\|_{Y^s_T}\\ &
\quad +\sum_{j=1}^4\sum_{m=1}^3\prod_{k=1 \atop k \neq j, m}^4\|u_k\|_{Y^s_T}
 \prod_{l=1}^3 \|u_{j,l}\|_{Y^s_T} \prod_{n=1}^3 \|u_{m,n}\|_{Y^s_T}\, .
\end{split}
\end{equation}
\end{proposition}

\begin{proof} Let $0<t \le T \le 1$.
First, assume that $N \le N_0=2^9$. By using the definition of $\mathcal{E}_N$ in \eqref{defEN}, we have
\begin{displaymath}
\frac{d}{dt} \mathcal{E}_N(t) =c_4 \int_{\mathbb R}P_N\partial_x\big( u_1u_2u_3 \big) P_Nu \, dx \, ,
\end{displaymath}
which yields after integrating between $0$ and $t$ and applying H\"older's inequality  that
\begin{displaymath}
\begin{split}
|\mathcal{E}_N(t)| &\le |\mathcal{E}_N(0)|+|c_4|\Big|\int_{\mathbb R_t}P_N\partial_x\big( u_1u_2u_3 \big) P_Nu  \, \Big| \\ &
\lesssim  |\mathcal{E}_N(0)|+\prod_{i=1}^4 \|u_i\|_{L^\infty_TL^{4}_x}\lesssim  \mathcal{E}_N(0)+\prod_{i=1}^4 \|u_i\|_{L^\infty_TH^{\frac14}_x}
\end{split}
\end{displaymath}
where the notation $\mathbb R_t=\mathbb R \times [0,t]$ defined at the beginning of Section \ref{Secmultest} has been used. Thus, we deduce after taking the supreme over $t \in [0,T]$ and summing over $N \le N_0$ (recall here that we use a nonhomogeneous dyadic decomposition in $N$) that
\begin{equation} \label{prop-ee.2}
\sum_{N \le N_0} N^{2s} \sup_{t \in [0,T]} \big|\mathcal{E}_N(t) \big| \lesssim \sum_{N \le N_0} N^{2s} \big|\mathcal{E}_N(0) \big|+\prod_{j=1}^4\|u_j\|_{Y_T^{\frac14}} \, .
\end{equation}

Next, we turn to the case where $N\ge N_0$. As above, we differentiate $\mathcal{E}_N$ with respect to time and then integrate between 0 and $t$ to get
\begin{align}
N^{2s}\mathcal{E}_N(t) &= N^{2s}\mathcal{E}_N(0) + c_4N^{2s}\int_{\R_t}P_N\partial_x(u_1u_2u_3)P_Nu + \alpha N^{2s}\int_0^t\frac{d}{dt}\mathcal{E}_N^{3}(t')dt' \nonumber \\
&\quad  + \beta N^{2s} \int_0^t \frac{d}{dt} \mathcal{E}_N^5(t')dt' \nonumber \\
&=: N^{2s}\mathcal{E}_N(0) + c_4I_N + \alpha J_N  + \beta K_N \, . \label{prop-ee.3}
\end{align}

We rewrite $I_N$ in Fourier variable and get
\begin{align*}
I_N &= N^{2s} \int_{\Gamma^3_t} (-i\xi_4) \phi_N^2(\xi_4) \widehat{u}_1(\xi_1) \widehat{u}_2(\xi_2) \widehat{u}_3(\xi_3) \widehat{u}_4(\xi_4) \\
&= \sum_{(M_1,M_2,M_3) \in \mathbb D^3} N^{2s} \int_{\Gamma^3_t} (-i\xi_4) \phi_{M_1,M_2,M_3}(\vec{\xi}_{(3)}) \phi_N^2(\xi_4) \prod_{j=1}^4\widehat{u}_j(\xi_j) \, .
\end{align*}
Next we decompose $I_N$ as
\begin{align}
  I_N &=  N^{2s}\left(\sum_{\mathcal{M}_3^{low}} + \sum_{\mathcal{M}_3^{med}} +\sum_{\mathcal{M}_3^{high}}\right) \int_{\Gamma^3_t} (-i\xi_4) \phi_{M_1,M_2,M_3}(\vec{\xi}_{(3)}) \phi_N^2(\xi_4) \prod_{j=1}^4\widehat{u}_j(\xi_j) \nonumber \\
  &=: I_N^{low} + I_N^{med} + I_N^{high} \, , \label{prop-ee.4}
\end{align}
by using the notations in Section \ref{Secmultest}.

\medskip

\noindent \textit{Estimate for $I_N^{low}$.}
Thanks to Lemma \ref{teclemma}, the integral in $I_N^{low}$ is non trivial for $|\xi_1|\sim |\xi_2|\sim |\xi_3|\sim |\xi_4|\sim N$ and $M_{min}\le N^{-\frac12}$. Therefore we get from Lemma \ref{prod4-est} that
\begin{equation*}
\begin{split}
  |I_N^{low}| &\lesssim \sum_{M_{min}\le N^{-\frac12} \atop M_{min} \le M_{med} \ll N} N^{2s+1}M_{min} \prod_{j=1}^4 \|P_{\sim N}u_j\|_{L^\infty_TL^2_x}
  \lesssim \prod_{j=1}^4 \|P_{\sim N}u_j\|_{L^\infty_TH^s_x} \, ,
\end{split}
\end{equation*}
since $(2s+\frac12)<4s$. This leads to
\begin{equation} \label{prop-ee.5}
\sum_{N \ge N_0}|I_N^{low}| \lesssim \prod_{j=1}^4 \|u_j\|_{Y^s_T} \, .
\end{equation}

\medskip

\noindent \textit{Estimate for $I_N^{high}$.} We perform  nonhomogeneous dyadic decompositions  $\displaystyle{u_j =\sum_{N_j} P_{N_j}u_j}$ for $j=1,2,3$. We assume without loss of generality that $ N_1=\max(N_1,N_2,N_3) $. Recall that this ensures that $M_{max}\sim N_1$.  We separate the contributions of two regions that we denote $I_N^{high,1}$ and $I_N^{high,2}$.\\
{$ \bullet \;  M_{min}\le N^{-1}$.}
Then  we apply Lemma \ref{prod4-est}  on the sum over $ M_{med} $ and use   the discrete Young's inequality to get
\begin{align}
|I_N^{high,1}| &\lesssim  \sum_{M_{min}\le N^{-1} } N^{2s+1}M_{min}\sum_{N_{1} \gtrsim N, N_2,N_3} \prod_{j=2}^3 \|P_{ N_j}u_j\|_{L^\infty_TL^2_x} \|P_{ N_1}u_1\|_{L^2_{T,x}} \|P_{ N}u_4\|_{L^2_{T,x}} \nonumber \\
 \lesssim & \sum_{N_{1}\ge N}
 \Bigl(\frac{N}{N_{1}}\Bigr)^{s}
  \|P_{N_1} u_1\|_{L^2_T H^{s}_x}   \|P_{N} u_4\|_{L^2_T H^{s}_x}  \|u_2\|_{L^\infty_T H^{0+}_x} \|u_3\|_{L^\infty_T H^{0+}_x} \nonumber\\
\lesssim & \, \delta_N  \|P_{N} u_4\|_{L^2_T H^{s}_x}  \prod_{i=1}^3 \|u_i\|_{L^\infty_T H^{s}_x} \; , \label{yoyo}
\end{align}
with $ \{\delta_{2^j}\}\in l^2(\mathbb N) $.
Summing over $N$  this leads to
\begin{equation} \label{prop-ee.5b}
\sum_{N \ge N_0}|I_N^{high,1}| \lesssim \prod_{j=1}^4 \|u_j\|_{Y^s_T} \, .
\end{equation}

\noindent
{$ \bullet \;  M_{min}> N^{-1}$.} For $j=1,\cdots,4$, let $\tilde{u}_j$ be  an extension of $u_j$ to $\mathbb R^2$ such that
 $\|\tilde{u}_j \|_{Y^s} \le 2 \|u \|_{Y^s_T}$.
 Now, we define $u_{N_j}=P_{N_j}\tilde{u}_j$ and perform nonhomogeneous dyadic decompositions in $N_j$, so that $I_N^{high,2}$ can be rewritten as
\begin{equation*}
I_N^{high,2} =N^{2s+1} \sum_{N_j, N_4 \sim N} \sum_{(M_1,M_2,M_3) \in \mathcal{M}_3^{high}} \int_{\mathbb R_t}
\Pi^3_{\eta,M_1,M_2,M_3}(u_{N_1},u_{N_2},u_{N_3}) \, u_{N_4} \, ,
\end{equation*}
with $\eta(\xi_1,\xi_2,\xi_3)= \phi_N^2(\xi_4)\frac{i\xi_4}{N} \in L^{\infty}(\Gamma^3)$.
Thus, it follows from \eqref{L2trilin.2} that
\begin{eqnarray*}
|I_N^{high,2}| &  \lesssim  & N^{2s} \sum_{N_j, N_4 \sim N} \frac{N}{N_{max}}
\Bigl( \sum_{N^{-1} < M_{min} \le 1 \atop N \lesssim M_{med} \le M_{max} \lesssim N_{max}}  M_{min}^\frac{1}{16}+
 \sum_{1 < M_{min} \lesssim N_{med} \atop N \lesssim M_{med} \le M_{max} \lesssim N_{max}}\Bigr)  \\
 & &\quad \|u_{N_1}\|_{Y^0} \|u_{N_2}\|_{Y^0} \|u_{N_3}\|_{Y^0} \|u_{N_4}\|_{Y^0}\, .
\end{eqnarray*}
Proceeding as in \eqref{yoyo} (here we sum over $ M_{min}\le 1  $ by using the factor $ M_{min}^\frac{1}{16}  $  and over  $ M_{min}\ge 1  $ by using that $ M_{min}\le N_{med} $ ) we get
\begin{equation} \label{prop-ee.6}
\sum_{N \ge N_0}|I_N^{high,2}| \lesssim \prod_{j=1}^4 \|u_j\|_{Y^s_T} \, .
\end{equation}

\medskip
\medskip

\noindent \textit{Estimate for $c_4 I_N^{med}+\alpha J_N+\beta K_N$.}
Using (\ref{eq-u0})-(\ref{eq-u1}), we can rewrite $\frac{d}{dt}\mathcal{E}_N^{3}$ as
\begin{align*}
    &\sum_{\mathcal{M}_3^{med}} \int_{\Gamma^3} \phi_{M_1,M_2,M_3}(\vec{\xi}_{(3)}) \phi_N^2(\xi_4) \frac{i\xi_4(\xi_1^3+\xi_2^3+\xi_3^3+\xi_4^3)}{\Omega^3(\vec{\xi}_{(3)})}  \prod_{j=1}^4\widehat{u}_j(\xi_j) \\
  &+ \sum_{j=1}^4 c_j \sum_{\mathcal{M}_3^{med}} \int_{\Gamma^3} \phi_{M_1,M_2,M_3}(\vec{\xi}_{(3)}) \phi_N^2(\xi_4)\frac{\xi_4}{\Omega^3(\vec{\xi}_{(3)})} \prod_{k=1 \atop{k \neq j}}^4\widehat{u}_k(\xi_k)  \mathcal{F}_x\partial_x \big( u_{j,1}u_{j,2}u_{j,3} \big)(\xi_j) \, .
  \end{align*}
Using (\ref{res3}), we see by choosing $\alpha=c_4$ that $I_N^{med}$ is canceled out by the first term of the above expression. Hence,
\begin{equation} \label{prop-ee.7}
c_4 I_N^{med}+\alpha J_N = c_4\sum_{j=1}^4 c_j J_N^j \, ,
\end{equation}
where, for $j=1,\cdots,4$,
\begin{displaymath}
J_N^j = iN^{2s}\sum_{\mathcal{M}_3^{med}}  \int_{\Gamma^5_t} \phi_{M_1,M_2,M_3}(\vec{\xi}_{(3)}) \phi_N^2(\xi_4) \frac{\xi_4\xi_j}{\Omega^3(\vec{\xi}_{(3)})} \prod_{k=1 \atop k\neq j}^4 \widehat{u}_k(\xi_k) \prod_{l=1}^3 \widehat{u}_{j,l}(\xi_{j,l}) \, ,
\end{displaymath}
with the convention $\displaystyle{\xi_j=-\sum_{k=1 \atop{k \neq j}}^4\xi_k=\sum_{l=1}^3\xi_{j,l}}$ and the notation $\vec{\xi}_{(3)}=(\xi_1,\xi_2,\xi_3)$.
Now, we define $\vec{\xi_j}_{(3)}$, for $j=1,2,3,4$ as follows:
\begin{displaymath}
\vec{\xi_1}_{(3)}=(\xi_2,\xi_3,\xi_4), \ \vec{\xi_2}_{(3)}=(\xi_1,\xi_3,\xi_4), \ \vec{\xi_3}_{(3)}=(\xi_1,\xi_2,\xi_4), \ \vec{\xi_4}_{(3)}=(\xi_1,\xi_2,\xi_3) \, .
\end{displaymath}
With this notation in hand and by using the symmetries of the functions $\sum_{\mathcal{M}_3^{med}}\phi_{M_1,M_2,M_3}$ and $\Omega^3$, we obtain that
\begin{displaymath}
J_N^j = iN^{2s}\sum_{\mathcal{M}_3^{med}}  \int_{\Gamma^5_t} \phi_{M_1,M_2,M_3}(\vec{\xi_j}_{(3)}) \phi_N^2(\xi_4) \frac{\xi_4\xi_j}{\Omega^3(\vec{\xi_j}_{(3)})} \prod_{ k=1\atop k\neq j}^4 \widehat{u}_k(\xi_k) \prod_{l=1}^3 \widehat{u}_{j,l}(\xi_{j,l}) \, .
\end{displaymath}

Moreover, observe from the definition of $\mathcal{M}_3^{med}$ in Section \ref{Secmultest} that
\begin{displaymath}
|\xi_1|\sim |\xi_2|\sim |\xi_3|\sim |\xi_4|\sim N \quad \text{and} \quad \left|\frac{\xi_j\xi_4}{\Omega^3(\vec{\xi}_{(3)})}\right| \sim \frac{N}{M_{min(3)}M_{med(3)}} \, ,
\end{displaymath}
on the integration domain of $J_N^j$. Here $M_{max(3)} \ge M_{med(3)} \ge M_{min(3)}$ denote the maximum, sub-maximum and minimum of $\{M_1,M_2,M_3\}$.

Since $\max(|\xi_{j,1}+\xi_{j,2}|, |\xi_{j,1}+\xi_{j,3}|, |\xi_{j,2}+\xi_{j,3}|) \gtrsim N$ on the integration domain of $J_N^j$, we may decompose $\sum_jc_jJ_N^j$ as
\begin{align}
  \sum_{j=1}^4 c_jJ_N^j &= iN^{2s}\left(\sum_{\mathcal{M}_5^{low}} + \sum_{\mathcal{M}_5^{med}} + \sum_{\mathcal{M}_5^{high}}\right) \sum_{j=1}^4c_j \nonumber \\
  &\quad \times \int_{\Gamma^5_t} \phi_{M_1,...,M_6}(\vec{\xi_j}_{(5)}) \phi_N^2(\xi_4) \frac{\xi_4\xi_j}{\Omega^3(\vec{\xi_j}_{(3)})} \prod_{k=1\atop k\neq j}^4 \widehat{u}_k(\xi_k) \prod_{l=1}^3 \widehat{u}_{j,l}(\xi_{j,l}) \nonumber \\
  &:= J_N^{low} + J_N^{med} + J_N^{high} \, , \label{prop-ee.8}
\end{align}
where $\vec{\xi_j}_{(5)}=(\vec{\xi_j}_{(3)},\xi_{j,1},\xi_{j,2},\xi_{j,3}) \in \Gamma^5$.

Moreover, we may assume by symmetry that $M_1 \le M_2 \le M_3$ and $M_4 \le M_5 \le M_6$.

\medskip

\noindent \textit{Estimate for $J^N_{low}$.} In the region $\mathcal{M}^{low}_5$, we have that $M_4 \lesssim M_2$.   Moreover, from Lemma \ref{teclemma}, the integral in $J_N^{low}$ is non trivial for $|\xi_1|\sim \cdots \sim |\xi_4|\sim N$, $|\xi_{j,1}|\sim |\xi_{j,2}| \sim |\xi_{j,3}| \sim N$ and $M_3 \sim M_6 \sim N$. Therefore by using \eqref{prod6.est.1}, we can bound $|J_N^{low}|$ by
\begin{align*}
   & \sum_{j=1}^4 \sum_{N^{-\frac12} < M_1\le M_2\ll N\atop}\sum_{M_4\lesssim M_2 \atop M_4 \le M_5 \ll N} N^{2s}M_1M_4 \frac{N}{M_1M_2} \prod_{k=1\atop k\neq j}^4 \|P_{\sim N}u_k\|_{L^\infty_TL^2_x} \prod_{l=1}^3 \|P_{\sim N}u_{j,l}\|_{L^\infty_TL^2_x}\\
  &\lesssim \sum_{j=1}^4 \prod_{k=1\atop k\neq j}^4 \|P_{\sim N}u_k\|_{L^\infty_TH^s_x} \prod_{l=1}^3 \|P_{\sim N}u_{j,l}\|_{L^\infty_TH^s_x} \, ,
\end{align*}
since $s>1/4$. Thus, we deduce that
\begin{equation} \label{prop-ee.9}
\sum_{N \ge N_0}|J_N^{low}| \lesssim  \sum_{j=1}^4 \prod_{k=1\atop k\neq j}^4 \|u_k\|_{Y^s_T} \prod_{l=1}^3 \|u_{j,l}\|_{Y^s_T} \, .
\end{equation}

\medskip

\noindent \textit{Estimate for $J_N^{high}$.} Proceeding as for $I_N^{high}$, we split $J_N^{high}$ into $J_N^{high,1}+J_N^{high,2}$ to separate the contributions depending on whether $M_4\le N^{-1}$ or $M_4>N^{-1}$.

\noindent
{$ \bullet \;  M_4\le N^{-1}$.} From Lemma \ref{teclemma}, the integral in $J_N^{high,1}$ is non trivial for $|\xi_1|\sim \cdots \sim |\xi_4|\sim N$, $M_3 \sim N$, $N_{max(5)}=\max\{N_{j,1}, N_{j,2}, N_{j,3}\} \gtrsim N$,  $M_4\le N^{-1} $ and $ M_5\sim M_6\sim N_{max(5)}$ . Therefore by using \eqref{prod6.est.1}, we can bound $|J_N^{high,1}|$ by
\begin{align*}
   & \sum_{j=1}^4 \sum_{N^{-\frac12} <M_1\le M_2\ll N \atop}\sum_{M_4\le N^{-1} } \sum_{N_{j,l}}  \frac{N^{2s+1}M_1M_4}{M_1M_2} \prod_{k=1\atop k\neq j}^4 \|P_{\sim N}u_k\|_{L^\infty_TL^2_x} \prod_{l=1}^3 \|P_{N_{j,l}}u_{j,l}\|_{L^\infty_TL^2_x}\\
  &\lesssim \sum_{j=1}^4 \prod_{k=1\atop k\neq j}^4 \|P_{\sim N}u_k\|_{L^\infty_TH^s_x} \prod_{l=1}^3 \|u_{j,l}\|_{L^\infty_TH^s_x} \, ,
\end{align*}
since $s>\frac14$. This leads to
\begin{equation} \label{prop-ee.9b}
\sum_{N \ge N_0}|J_N^{high,1}| \lesssim  \sum_{j=1}^4 \prod_{k=1\atop k\neq j}^4 \|u_k\|_{Y^s_T} \prod_{l=1}^3 \|u_{j,l}\|_{Y^s_T} \, .
\end{equation}

\medskip

\noindent
{$ \bullet \;  M_4> N^{-1}$.} For $1 \le k \le 4$, and $1 \le l \le 3$ let $\tilde{u}_k$ and $\tilde{u}_{j,l}$ be  suitable extensions of $u_k$ and $u_{j,l}$ to $\mathbb R^2$. We define $u_{N_k}=P_{N_k}\tilde{u}_k$, $u_{N_{j,l}}=P_{N_{j,l}}\tilde{u}_{j,l}$ and perform nonhomogeneous dyadic decompositions in $N_k$ and $N_{j,l}$.

We first estimate $J_N^{high,2}$ in the resonant case $M_1M_2M_3\sim M_4M_5M_6$. We assume to simplify the notations that $M_1\le M_2\le M_3$ and $M_4\le M_5\le M_6$. Since we are in $\mathcal{M}_5^{high}$, we have that $M_5,M_6\gtrsim N$ and $M_1,M_2\ll N$ which yields
$$
M_3\sim N \quad \text{and} \quad M_4\sim \frac{M_1M_2N}{M_5M_6}\ll N \, .
$$
This forces $ N_{j,1} \sim N$ and it follows from \eqref{prod6.est.2} that
\begin{align*}
  |J_N^{high,2}| &\lesssim \sum_{j=1}^4\sum_{\mathcal{M}^{high}_5}\sum_{N_{j,l}}\frac{N^{2s+1}}{M_1M_2} M_1M_4^{\frac12}N_{j,2}^{\frac14}N_{j,3}^{\frac14} \prod_{k=1\atop k\neq j}^4 \|P_{\sim N} u_k\|_{L^\infty_TL^2_x} \prod_{l=1}^3\|u_{N_{j,l}}\|_{L^\infty_TL^2_x} \\
  &\lesssim \sum_{j=1}^4 \sum_{N^{-\frac12} \le M_1 \le M_2 \ll N \atop}N^{s+\frac12}\frac{(M_1M_2)^{\frac12}}{M_2} \prod_{k=1\atop k\neq j}^4 \|P_{\sim N} u_k\|_{L^\infty_TL^2_x} \prod_{l=1}^3
  \|u_{j,l}\|_{L^\infty_TH^s_x} \, .
\end{align*}
Summing over $N^{-1/2} \le M_1, \ M_2 \ll N$ and $N \ge N_0$ and using the assumption $s>\frac14$, we get
\begin{equation} \label{prop-ee.10}
\sum_{N \ge N_0}|J_N^{high,2}| \lesssim  \sum_{j=1}^4 \prod_{k=1\atop k\neq j}^4 \|u_k\|_{Y^s_T} \prod_{l=1}^3 \|u_{j,l}\|_{Y^s_T} \, ,
\end{equation}
in the resonant case.

By using \eqref{L25lin.2}, we easily estimate $J_N^{high,2}$ in the non resonant case $M_1M_2M_3\not\sim M_4M_5M_6$ by
\begin{displaymath}
\begin{split}
|J_N^{high,2}| &\lesssim \sum_{j=1}^4\sum_{N^{-\frac12} \le M_1\le M_2 \ll N \atop } \sum_{N^{-1} < M_4 \le N \atop N \lesssim M_5 \le M_6 \lesssim N_{max(5)}}\sum_{N_{j,l}}\\ & \quad \quad \times \frac{N^{2s+1}}{M_1M_2} M_1N_{max(5)}^{-1} \prod_{k=1\atop k\neq j}^4 \| P_{\sim N} \tilde{u}_{k}\|_{Y^0} \prod_{k=1}^3 \|u_{N_{j,l}}\|_{Y^0} \, .
\end{split}
\end{displaymath}
Recalling that $N_{max(5)}=\max\{N_{j,1},N_{j,2},N_{j,3}\} \gtrsim N$, we conclude after summing over $N$  that \eqref{prop-ee.10} also holds, for $ s>\frac14$, in the non resonant case.

\medskip

\medskip

\noindent \textit{Estimate for $\alpha J_N^{med}+\beta K_N$}. Using equations (\ref{eq-u0})-(\ref{eq-u1})-(\ref{eq-u2}) and the resonance relation \eqref{res5}, we can rewrite $N^{2s}\int_0^t\frac{d}{dt}\mathcal{E}_N^5dt$ as
\begin{align*}
&N^{2s}\sum_{\mathcal{M}_5^{med}}\sum_{j=1}^4 c_j \int_{\Gamma^5_t}\phi_{M_1,...,M_6}(\vec{\xi_j}_{(5)}) \phi_N^2(\xi_4) \frac{i\xi_4\xi_j}{\Omega^3(\vec{\xi_j}_{(3)})} \prod_{k=1\atop k\neq j}^4 \widehat{u}_k(\xi_k) \prod_{l=1}^3 \widehat{u}_{j,l}(\xi_{j,l})\\
  &+ N^{2s}\sum_{\mathcal{M}_5^{med}}\sum_{j=1}^4 c_j \sum_{m=1\atop m\neq j}^4 c_m \int_{\Gamma^5_t}\phi_{M_1,...,M_6}(\vec{\xi_j}_{(5)}) \phi_N^2(\xi_4) \frac{\xi_4\xi_j}{\Omega^3(\vec{\xi_j}_{(3)})\Omega^5(\vec{\xi_j}_{(5)}) } \\
  &\quad\quad \times \prod_{k=1\atop k\neq j,m}^4 \widehat{u}_k(\xi_k) \,  \mathcal{F}_x\partial_x(u_{m,1}u_{m,2}u_{m,3})(\xi_m) \prod_{l=1}^3 \widehat{u}_{j,l}(\xi_{j,l})\\
  &+ N^{2s}\sum_{\mathcal{M}_5^{med}}\sum_{j=1}^4 c_j \sum_{m=1}^3 c_{j,m} \int_{\Gamma^5_t}\phi_{M_1,...,M_6}(\vec{\xi_j}_{(5)}) \phi_N^2(\xi_4) \frac{\xi_4\xi_j}{\Omega^3(\vec{\xi_j}_{(3)})\Omega^5(\vec{\xi_j}_{(5)}) }\\
  &\quad\quad \times \prod_{k=1\atop k\neq j}^4 \widehat{u}_k(\xi_k)  \prod_{l=1\atop  l\neq m}^3 \widehat{u}_{j,l}(\xi_{j,l}) \, \mathcal{F}_x\partial_x(u_{j,m,1}u_{j,m,2}u_{j,m,3})(\xi_{j,m})
\\
  &:= K_N^1+K_N^2+K_N^3.
\end{align*}
By choosing $\beta=-\alpha$, we have that
\begin{equation} \label{prop-ee.11}
\alpha J_N^{med} + \beta K_N = \beta(K_N^2+K_N^3) \, .
\end{equation}

For the sake of simplicity, we will only consider the contribution of $K_N^3$ corresponding to a fixed $(j,m) \in \{1,2,3,4\} \times \{1,2,3\}$, since the other contributions on the right-hand side of \eqref{prop-ee.11} can be treated similarly.

Thus, for $(j,m)$ fixed, we need to bound
\begin{displaymath}
\tilde{K}_N:=  iN^{2s}\sum_{\mathcal{M}_5^{med}} \int_{\Gamma^7_t} \sigma(\vec{\xi_j}_{(5)})  \prod_{k=1\atop k\neq j}^4 \widehat{u}_k(\xi_k)  \prod_{l=1\atop  l\neq m}^3 \widehat{u}_{j,l}(\xi_{j,l}) \prod_{n=1}^3\widehat{u}_{j,m,n}(\xi_{j,m,n}) \, ,
\end{displaymath}
with the conventions $\displaystyle{\xi_j=-\sum_{k=1 \atop k \neq j}^4\xi_k=\sum_{l=1}^3\xi_{j,l}}$ and $\displaystyle{\xi_{j,m}=\sum_{n=1}^3\xi_{j,m,n}}$ and where
\begin{displaymath}
\sigma(\vec{\xi_j}_{(5)}) = \phi_{M_1,...,M_6}(\vec{\xi_j}_{(5)}) \, \phi_N^2(\xi_4) \, \frac{\xi_4 \, \xi_j \, \xi_{j,m}}{\Omega^3(\vec{\xi_j}_{(3)}) \, \Omega^5(\vec{\xi_j}_{(5)}) } \, .
\end{displaymath}

Now, let us define $\vec{\xi}_{j,m_{(7)}} \in \Gamma^7$ as follows:
\begin{align*}
\vec{\xi}_{j,1_{(7)}}&=\big(\vec{\xi_j}_{(3)},\xi_{j,2}, \xi_{j,3},\xi_{j,1,1},\xi_{j,1,2},\xi_{j,1,3}\big) \, ,\\ \vec{\xi}_{j,2_{(7)}}&=\big(\vec{\xi_j}_{(3)},\xi_{j,1},\xi_{j,3},\xi_{j,2,1},\xi_{j,2,2},\xi_{j,2,3}\big) \, ,
\\ \vec{\xi}_{j,3_{(7)}}&=\big(\vec{\xi_j}_{(3)},\xi_{j,1},\xi_{j,2},\xi_{j,3,1},\xi_{j,3,2},\xi_{j,3,3}\big) \, .
\end{align*}

We decompose $\tilde{K}_N$ as
\begin{align}
  \tilde{K}_N &= iN^{2s} \left(\sum_{\mathcal{M}_7^{low}} + \sum_{\mathcal{M}_7^{med}}+\sum_{\mathcal{M}_7^{high}}\right)
  \int_{\Gamma^7_t} \widetilde{\sigma}(\vec{\xi}_{{j,m}_{(7)}})  \prod_{k=1\atop k\neq j}^4 \widehat{u}_k(\xi_k)  \prod_{l=1\atop  l\neq m}^3 \widehat{u}_{j,l}(\xi_{j,l}) \prod_{n=1}^3\widehat{u}_{j,m,n}(\xi_{j,m,n}) \nonumber\\
  &=\tilde{K}_N^{low} +\tilde{K}_N^{med}+ \tilde{K}_N^{high} \, ,  \label{prop-ee.12}
\end{align}
where
\begin{displaymath}
\widetilde{\sigma}(\vec{\xi}_{{j,m}_{(7)}}) = \phi_{M_7,M_8,M_9}\big(\xi_{j,m,1},\xi_{j,m,2},\xi_{j,m,3}\big) \, \sigma(\vec{\xi_j}_{(5)}) \, .
\end{displaymath}

Observe from Lemma \ref{teclemma} that the integrand is non trivial for
\begin{displaymath}
|\xi_1|\sim\cdots \sim |\xi_4|\sim |\xi_{j,1}| \sim |\xi_{j,2}| \sim |\xi_{j,3}| \sim |\xi_{j,m,1}+\xi_{j,m,2}+\xi_{j,m,3}| \sim N \, .
\end{displaymath}
Moreover, we have
\begin{displaymath}
M_{max(3)} \sim M_{max(5)} \sim N \quad \text{and}  \quad N^{-\frac12} \le M_{min(3)}\le M_{med(3)} \le M_{min(5)} \le M_{med(5)} \ll N \, .
\end{displaymath}
Hence,
\begin{displaymath}
\big|\widetilde{\sigma}(\vec{\xi}_{{j,m}_{(7)}})\big| \sim \frac{N}{{M_{min(3)}M_{med(3)}M_{min(5)}M_{med(5)}}} \, .
\end{displaymath}
Note that we can always assume by symmetry and without loss of generality that $M_1 \le M_2 \le M_3$ and $M_7 \le M_8 \le M_9$.

\medskip
\noindent \textit{Estimate for $\tilde{K}_N^{low}$.} In the integration domain of $\tilde{K}_N^{low}$ we have from Lemma \ref{teclemma} that $|\xi_{j,m,1}|\sim |\xi_{j,m,2}|\sim |\xi_{j,m,3}|\sim N$.

Then it follows applying  \eqref{prod8-est.1} or \eqref{prod8-est.100} (depending on wether $M_{min(5)}=M_6$ or $M_{min(5)}=M_4$ or $M_5$) on the sum over $ (M_8,M_9) $ that
\begin{align*}
|\tilde{K}_N^{low}| &\lesssim \sum_{N^{-\frac12}<M_1 \le M_2\ll N\atop M_2 \ll M_{min(5)} \le M_{med(5)}\ll N}\sum_{M_7 \lesssim M_{med(5)}} \frac{N^{2s+1}M_7}{M_2M_{med(5)}} \\ & \quad \quad \times \prod_{k=1 \atop k \neq j}^4\|P_{\sim N}u_k\|_{L^\infty_TL^2_x} \prod_{l=1 \atop l\neq m}^3 \|P_{\sim N}u_{j,l}\|_{L^\infty_TL^2_x} \prod_{n=1}^3\|P_{\sim N}u_{j,m,n}\|_{L^\infty_TL^2_x} \, .
\end{align*}
This implies that
\begin{equation} \label{prop-ee.13}
\sum_{N \ge N_0}|\tilde{K}_N^{low}| \lesssim \prod_{k=1 \atop k \neq j}^4\|u_k\|_{L^\infty_TH^s_x} \prod_{l=1 \atop l\neq m}^3 \|u_{j,l}\|_{L^\infty_TH^s_x} \prod_{n=1}^3\|u_{j,m,n}\|_{L^\infty_TH^s_x} \, ,
\end{equation}
since $2s+\frac32<8s \Leftrightarrow s>\frac14$.

\medskip
\noindent \textit{Estimate for $\tilde{K}_N^{med}$.}
In the integration domain of $\tilde{K}_N^{med}$ we have from Lemma \ref{teclemma} that $|\xi_{j,m,1}|\sim |\xi_{j,m,2}|\sim |\xi_{j,m,3}|\sim N$.
To estimate $\tilde{K}_N^{med}$, we divide the regions where $M_7 \le 1$ and $M_7 \ge 1$.

In the region where $M_7 \le 1$, we deduce by using
\eqref{prod8-est.1} or \eqref{prod8-est.100} (depending on wether $M_{min(5)}=M_6$ or $M_{min(5)}=M_4$ or $M_5$) on the sum over $ (M_8,M_9) $ that
\begin{align*}
|\tilde{K}_N^{med}| &\lesssim \sum_{N^{-\frac12}<M_1 \le M_2\ll N\atop M_2 \ll M_{min(5)} \le M_{med(5)}\ll N}\sum_{M_7 \le 1} \frac{N^{2s+1}M_7}{M_2M_{med(5)}} \\ & \quad \quad \times \prod_{k=1 \atop k \neq j}^4\|P_{\sim N}u_k\|_{L^\infty_TL^2_x} \prod_{l=1 \atop l\neq m}^3 \|P_{\sim N}u_{j,l}\|_{L^\infty_TL^2_x} \prod_{n=1}^3\|P_{\sim N}u_{j,m,n}\|_{L^\infty_TL^2_x} \, .
\end{align*}
This implies that
\begin{equation} \label{prop-ee.130}
\sum_{N \ge N_0}|\tilde{K}_N^{med}| \lesssim \prod_{k=1 \atop k \neq j}^4\|u_k\|_{L^\infty_TH^s_x} \prod_{l=1 \atop l\neq m}^3 \|u_{j,l}\|_{L^\infty_TH^s_x} \prod_{n=1}^3\|u_{j,m,n}\|_{L^\infty_TH^s_x} \, ,
\end{equation}
since $2s+2<8s \Leftrightarrow s>\frac13$.

In the region where $M_7 \ge 1$, for $1 \le k \le 4$, $k \neq j$, $1 \le l \le 3$, $l \neq m$ and $1 \le n \le 3$ let $\tilde{u}_k$, $\tilde{u}_{j,l}$ and $\tilde{u}_{j,m,n}$ be suitable extensions of $u_k$, $u_{j,l}$ and $u_{j,m,n}$ to $\mathbb R^2$. Then, we deduce from Lemma  \ref{teclemma} and \eqref{L27lin.200} that
\begin{align*}
|\tilde{K}_N^{med}| &\lesssim \sum_{N^{-\frac12}<M_1 \le M_2\ll N\atop M_2 \ll M_{min(5)} \le M_{med(5)}\ll N}\sum_{1 \le M_7 \le M_8} \frac{N^{2s+1}}{M_2M_{med(5)}M_8} \\ & \quad \quad \times \prod_{k=1 \atop k \neq j}^4  \|P_{\sim N} \widetilde{u}_k\|_{Y^0}\prod_{l=1 \atop l\neq m}^3
 \|P_{\sim N} \widetilde{u}_{j,l}\|_{Y^0}
  \prod_{n=1}^3 \|P_{\sim N} \widetilde{u}_{j,m,n}\|_{Y^0} \, .
\end{align*}
This implies that
\begin{equation} \label{prop-ee.1300}
\sum_{N \ge N_0}|\tilde{K}_N^{med}| \lesssim \prod_{k=1 \atop k \neq j}^4\|u_k\|_{Y_T^s} \prod_{l=1 \atop l\neq m}^3 \|u_{j,l}\|_{Y_T^s} \prod_{n=1}^3\|u_{j,m,n}\|_{Y_T^s} \, ,
\end{equation}
since $s>\frac13$.

\medskip
\noindent \textit{Estimate for $\tilde{K}_N^{high}$.}
We first estimate $\tilde{K}_N^{high}$ in the resonant case $M_4M_5M_6\sim M_7M_8M_9$. Since we are in $\mathcal{M}_7^{high}$, we have that $M_9\ge  M_8\gtrsim N$ and $M_{min(5)}\le  M_{med(5)}\ll N$. It follows that $M_{max(5)}\sim N$ and
\begin{displaymath}
M_7\sim \frac{M_{min(5)}M_{med(5)}N}{M_8M_9}\ll N \, .
\end{displaymath}
This forces $N_{j,m,1} \sim N$ (for example) and we deduce by using \eqref{prod8-est.0b} in the case $M_{min(5)}=M_6$, and \eqref{prod8-est.200} in the case $M_{min(5)}=M_4$ or $M_5$, that
\begin{align*}
  |\tilde{K}_N^{high}| &\lesssim \sum_{N^{-\frac12}<M_1 \le M_2 \ll N \atop M_2 \ll M_{min(5)} \le M_{med(5)} \ll N}\sum_{M_9\ge M_8\gtrsim N} \sum_{N_{j,m,n}, N_{j,m,1} \sim N\atop}\frac{N^{2s+\frac32}}{M_2M_8^{\frac12}M_9^{\frac12}} N_{j,m,2}^{\frac14}N_{j,m,3}^{\frac14}\\
   &\quad \times \prod_{k=1 \atop k \neq j}^4 \|P_{\sim N}u_k\|_{L^\infty_TL^2_x} \prod_{l=1 \atop l\neq m}^3\|P_{\sim N}u_{j,l}\|_{L^\infty_TL^2_x} \prod_{n=1}^3\|P_{N_{j,m,n}} u_{j,m,n}\|_{L^{\infty}_TL^2_x} \, ,
\end{align*}
which yields summing over  $N \ge N_0$ and using the assumption $s>\frac14$  that
\begin{equation} \label{prop-ee.14}
  \sum_{N \ge N_0}|\tilde{K}_N^{high}| \lesssim \prod_{k=1 \atop k \neq j}^4\|u_k\|_{Y^s_T} \prod_{l=1 \atop l\neq m}^3 \|u_{j,l}\|_{Y^s_T} \prod_{n=1}^3\|u_{j,m,n}\|_{Y^s_T} \, .
\end{equation}
Now, in the non resonant case we separate the contributions of the region $ M_7\le N^{-1} $ and $ M_7>N^{-1}$.
In the first region,
applying  \eqref{prod8-est.1} or \eqref{prod8-est.100} (depending on wether $M_{min(5)}=M_6$ or $M_{min(5)}=M_4$ or $M_5$) on the sum over $ (M_8,M_9)$, we get
\begin{align*}
  |\tilde{K}_N^{high}| &\lesssim \sum_{N^{-\frac12}<M_1 \le M_2 \ll N \atop M_2 \ll M_{min(5)} \le M_{med(5)} \ll N} \sum_{M_7 \le N^{-1}}\sum_{N_{j,m,n}}\frac{N^{2s+1}M_7}{M_2M_{med(5)}} \\
   &\quad \times \prod_{k=1 \atop k \neq j}^4 \|P_{\sim N}u_k\|_{L^\infty_TL^2_x} \prod_{l=1 \atop l\neq m}^3\|P_{\sim N}u_{j,l}\|_{L^\infty_TL^2_x} \prod_{n=1}^3\|P_{N_{j,m,n}}u_{j,m,n}\|_{L^{\infty}_TL^2_x} \, .
\end{align*}
Observing that $\max\{N_{j,m,1},N_{j,m,2},N_{j,m,3} \} \gtrsim N$, we conclude after summing over $N \ge N_0$ that
\begin{equation} \label{prop-ee.13b}
\sum_{N \ge N_0}|\tilde{K}_N^{high}| \lesssim \prod_{k=1 \atop k \neq j}^4\|u_k\|_{L^\infty_TH^s_x} \prod_{l=1 \atop l\neq m}^3 \|u_{j,l}\|_{L^\infty_TH^s_x} \prod_{n=1}^3\|u_{j,m,n}\|_{L^\infty_TH^s_x} \, ,
\end{equation}
since $2s+1<6s \Leftrightarrow s>\frac14$.

Finally we treat contribution of the region $ M_7>N^{-1}$.
 For $1 \le k \le 4$, $k \neq j$ $1 \le l \le 3$, $l \neq m$ and $1 \le n \le 3$ let $\tilde{u}_k$, $\tilde{u}_{j,l}$ and $\tilde{u}_{j,m,n}$ be suitable extensions of $u_k$, $u_{j,l}$ and $u_{j,m,n}$ to $\mathbb R^2$. We define $u_{N_k}=P_{N_k}\tilde{u}_k$, $u_{N_{j,l}}=P_{N_{j,l}}\tilde{u}_{j,l}$, $u_{N_{j,m,n}}=P_{N_{j,m,n}}\tilde{u}_{j,m,n}$ and perform nonhomogeneous dyadic decompositions in $N_k$, $N_{j,l}$ and $N_{j,m,n}$.
By using \eqref{L27lin.2}, we  estimate $\tilde{K}_N^{high}$ on this region by
\begin{displaymath}
\begin{split}
|\tilde{K}_N^{high}| &\lesssim \sum_{N^{-\frac12}<M_1 \le M_2 \ll N \atop M_2 \ll M_{min(5)} \le M_{med(5)} \ll N} \sum_{N^{-1} \le M_7 \le M_8\le M_9\lesssim N_{max(7)}}\sum_{N_k \sim N}\sum_{N_{j,l} \sim N} \sum_{N_{j,m,n}} \\ & \quad \times \frac{N^{2s+1}}{M_2M_{med(5)}} N_{max(7)}^{-1} \prod_{k=1 \atop k \neq j}^4\|u_{N_k}\|_{Y^0} \prod_{l=1 \atop l\neq m}^3 \|u_{N_{j,l}}\|_{Y^0} \prod_{n=1}^3\|u_{N_{j,m,n}}\|_{Y^0} \, ,
\end{split}
\end{displaymath}
where $N_{max(7)}=\max\{N_{j,m,1},N_{j,m,2},N_{j,m,3} \} \gtrsim N $.
Therefore,  \eqref{prop-ee.14} also holds, for $ s>\frac14$, in this region.

\medskip
Finally, we conclude the proof of Proposition \ref{prop-ee} gathering \eqref{prop-ee.2}--\eqref{prop-ee.14}.
\end{proof}

\begin{remark}
The restriction $s>\frac13$ only appears when estimating the contribution $\widetilde{K}_N^{med}$. All the other contributions are estimated with $s>\frac14$. It is likely that the index $\frac13$ may be improved by adding higher order modifications to the energy.
\end{remark}

\subsection{Estimates for the $X^{s-1,1}_T$ and  $X^{s-\frac{7}{8},\frac{15}{16}}_T$ norms}
In this subsection, we explain how to control the $X^{s-1,1}_T$ and  $X^{s-\frac{7}{8},\frac{15}{16}}_T$ norms that we used in the energy estimates.

We start by deriving   a suitable Strichartz estimate for the solutions of \eqref{eq-u0}.
\begin{proposition} \label{se}
Assume that $0<T \le 1$ and let $u\in L^\infty(]0,T[ \, : H^{\frac14}(\R))$ be a solution to \eqref{eq-u0} with $u_i\in  L^\infty(]0,T[ \, : H^{\frac14}(\R))$, $i=1,2,3$. Then,
\begin{equation} \label{se.1}
\|J_x^\frac{1}{7} u\|_{L^4_TL^{\infty}_x} \lesssim \|u\|_{L^{\infty}_TH^{\frac14}_x}+\prod_{j=1}^3\|u_j\|_{L^{\infty}_TH^{\frac14}_x} \, .
\end{equation}
\end{proposition}

\begin{proof}
Since $J_x^\frac{1}{7} u$ is a solution to \eqref{eq-u0} where we apply $ J_x^\frac{1}{7}$ on the RHS member, we use estimate \eqref{refinedStrichartz1}  with
  $F=J_x^\frac{1}{7}\partial_x(u_1u_2u_3)$ and $\delta=\frac{9}{7}+$. H\"older and Sobolev inequalities  then lead to
\begin{displaymath}
\begin{split}
\|J_x^\frac{1}{7} u\|_{L^4_TL^{\infty}_x} &\lesssim \|u\|_{L^{\infty}_TH^{\frac{3}{14}+}_x}+\|u_1u_2u_3\|_{L^4_TL^1_x} \lesssim \|u\|_{L^{\infty}_TH^\frac{1}{4}_x}+\prod_{j=1}^3\|u_j\|_{L^{\infty}_TH^\frac{1}{6}_x} \, .
\end{split}
\end{displaymath}
\end{proof}
The following proposition ensures that a  $  \widetilde{L^\infty_T}H^s $-solution to  \eqref{eq-u0} belongs to $ Y_T^s $.
\begin{proposition} \label{trilin}
Let $0<T \le 1$,  $s \ge \frac14$ and let $u, \, u_i, \, u_{i,j}, \, u_{i,j,k} \in \widetilde{L^\infty}(]0,T[ \, : H^s(\R)) $, $1 \le i,j,k \le 3$, be solutions to  \eqref{eq-u0}-\eqref{eq-u1}-\eqref{eq-u2} such . Then $ u\in Y_T^s $ and it holds
 \begin{eqnarray}
\|u\|_{Y^s_T} & \lesssim & \|u\|_{\widetilde{L^\infty_T} H^{s}}+\prod_{i=1}^3\|u_i\|_{L^{\infty}_TH^s_x}+\sum_{i=1}^3 \prod_{j=1 \atop j\neq i}^3
\Bigl( \|u_j\|_{L^{\infty}_TH^{\frac14}_x}+\prod_{k=1}^3\|u_{j,k}\|_{L^{\infty}_TH^{\frac14}_x}\Bigr)\|u_i\|_{L^{\infty}_T H^s_x}  \nonumber \\
& & +\sum_{i=1}^3 \prod_{j=1 \atop j\neq i}^3\| u_j\|_{L^\infty_T H^s_x} \Bigl[ \|u_i\|_{L^\infty_T H^{s}} +\sum_{k=1}^3 \prod_{l=1 \atop l\neq i}^3
\Bigl( \|u_{i,l}\|_{L^{\infty}_TH^{\frac14}_x} \nonumber \\ &&\quad \quad \quad \quad \quad \quad \quad \quad \quad \quad \quad \quad  +\prod_{m=1}^3\|u_{i,l,m}\|_{L^{\infty}_TH^{\frac14}_x}\Bigr)\|u_{i,k}\|_{L^{\infty}_T H^s_x} \Bigr]\, .
 \label{trilin.1}
\end{eqnarray}
\end{proposition}

\begin{proof}
In order to prove \eqref{trilin.1}, we have to extend the function $u$ from $ ]0,T[ $ to $ \R $. For this we use the extension operator $ \rho_T $ defined in Lemma \ref{extension}. In view of \eqref{resolution.2}, it remains to control the $ X^{s-1,1}_T $ and $X^{s-\frac{7}{8},\frac{15}{16}}_T$ norms of $ u $ to prove \eqref{trilin.1}.
 We claim that
 \begin{equation} \label{trilin.2}
\|u\|_{X^{s-1,1}_T} \lesssim \|u\|_{L^\infty_T H^{s}_x}+\sum_{i=1}^3 \prod_{j=1 \atop j\neq i}^3\|u_j\|_{L^4_TL^{\infty}_x}\|J^s_xu_i\|_{L^{\infty}_TL^2_x} \,
\end{equation}
and
\begin{eqnarray}
\|u\|_{X^{s-\frac{7}{8},\frac{15}{16}}_T} &\lesssim &  \|u\|_{L^\infty_T H^{s}_x}+\prod_{i=1}^3\|u_i\|_{L^{\infty}_TH^s_x}
+\sum_{i=1}^3 \prod_{j=1 \atop j\neq i}^3\|J_x^\frac{1}{7} u_j\|_{L^4_TL^{\infty}_x}\|J^s_xu_i\|_{L^{\infty}_TL^2_x} \nonumber \\
& & +\sum_{i=1}^3 \prod_{j=1 \atop j\neq i}^3\| u_j\|_{L^\infty_T H^s_x}\|u_i\|_{X^{s-1,1}_T}\,  . \label{trilin.3}
\end{eqnarray}
Noticing that \eqref{trilin.2} also holds for $ u_{k/k=1,2,3} $ with $ u_{l/l=1,2,3}$ replaced by $ u_{k,l} $ in the RHS member, these estimates together  with Proposition \ref{se} lead to  \eqref{trilin.1}.

 We start by proving \eqref{trilin.2}. Consider $\widetilde{u}=\rho_T(u)$ and $\widetilde{u}_i=\rho_T(u_i)$, $i=1,2,3$, the extensions of $u$ and $u_i$, $i=1,2,3$, to $\mathbb R^2$.
Recall the classical estimate
\begin{equation} \label{KatoPonce}
\|fg\|_{H^s} \lesssim \|f\|_{H^s}\|g\|_{L^{\infty}}+\|f\|_{L^{\infty}}\|g\|_{H^s} \, ,
\end{equation}
which holds for all $s \ge 0$, and can be found for instance in \cite{KaPo}.
By using this estimate, the Duhamel formula associated to \eqref{eq-u0} and the standard linear estimates in Bourgain's spaces (c.f. \cite{Bo1}), we get that
\begin{equation} \label{be}
\begin{split}
\|u\|_{X^{s-1,1}_T} \le \|\widetilde{u}\|_{X^{s-1,1}} &\lesssim \|u_0\|_{H^{s-1}}+ \|\partial_x(\widetilde{u}_1\widetilde{u}_2\widetilde{u}_3)\|_{X^{s-1,0}} \\
&  \lesssim \|u_0\|_{H^{s-1}}+\|J_x^s(\widetilde{u}_1\widetilde{u}_2\widetilde{u}_3)\|_{L^2_{x,t}}  \\
& \lesssim \|u\|_{L^\infty_T H^{s-1}_x}+\sum_{i=1}^3 \prod_{j=1 \atop j\neq i}^3\|\widetilde{u}_j\|_{L^4_tL^{\infty}_x}\|J^s_x\widetilde{u}_i\|_{L^{\infty}_tL^2_x} \, ,
\end{split}
\end{equation}
since, according to Remark \ref{rem2}, $u\in C([0,T]; H^{s-1}(\R))$. Therefore, estimate \eqref{trilin.2} follows from \eqref{be}, \eqref{tildenorm} and \eqref{extension.1}.

Let us now tackle  \eqref{trilin.3}. First, as above we have
$$
\|u\|_{X^{s-\frac{7}{8},\frac{15}{16}}_T} \lesssim \|u\|_{L^\infty_T H_x^{s-\frac{7}{8}}} + \| \widetilde{u}_1 \widetilde{u}_2 \widetilde{u}_3  \|_{X_T^{s+\frac{1}{8}, -\frac{1}{16}}}
$$
and it thus suffices to bound
$$
I:=
 \Bigl\|\frac{\langle \xi \rangle^{s+\frac{1}{8}} {\mathcal F}_{t,x} \Bigl( \tilde{u}_1 \tilde{u}_2 \tilde{u}_3\Bigr)}{\langle \tau-\xi^3\rangle^\frac{1}{16}}   \Bigr\|_{L^2(\R^2)}
 $$
 where $\tilde{u}_i=\rho_T(u_i) $, $i=1,2,3$. In the sequel, we drop the tilda to simplify the expression.

 We separate different regions of integration.\\
 {\bf 1.} $|\xi|\le 2^9 $.  The contribution of this region is easily estimated by
 $$
 I\lesssim \prod_{i=1}^3 \|u_i\|_{L^\infty_t L^3_x} \lesssim  \prod_{i=1}^3 \|u_i\|_{L^\infty_t H^{\frac{1}{6}}_x}\; .
 $$
 {\bf 2.}$ |\xi|>2^9 $.\\
 {\bf 2.1} $ |\tau-\xi^3|\ge \frac{\xi^2}{6}  $. By using \eqref{KatoPonce}, the contribution of this region is estimated by
 \begin{eqnarray*}
 I
 & \lesssim & \|u_1 u_2 u_3\|_{L^2_{t}H^s_x} \lesssim  \| \|u_1 u_2 u_3\|_{H^s_x} \|_{L^2_{t}}\\
  &\lesssim & \Bigl\|  \sum_{i=1}^3 \|u_i\|_{H^s_x} \prod_{j=1 \atop j\neq i}^3 \| u_j\|_{L^\infty_x} \Bigr\|_{L^2_t} \\
  & \lesssim &  \sum_{i=1}^3\|u_i\|_{L^\infty_t H^s_x} \prod_{j=1 \atop j\neq i}^3 \| u_j\|_{L^4_t L^\infty_x} \; .
 \end{eqnarray*}
 {\bf 2.2} $ |\tau-\xi^3|<\frac{\xi^2}{6} $.
 We perform  nonhomogeneous dyadic decompositions  $\displaystyle{u_j =\sum_{N_j\ge 0} P_{N_j}u_j}$ with $j=1,2,3$. We assume without loss of generality that $ N_1\ge N_2\ge N_3 $.\\
 {\bf 2.2.1} $ N_1\sim N_2$.
  \begin{eqnarray*}
  I & \lesssim & \sum_{N> 2^9}N^{s+\frac{1}{8}}\sum_{N_1\sim N_2 \gtrsim N} \sum_{N_3\ge 0}  \|P_{N} (P_{N_1} u_1 P_{N_2} u_2 P_{N_3} u_3)\|_{L^2_{tx}}\\
   & \lesssim &\sum_{N> 2^9}\sum_{N_1\sim N_2 \gtrsim N} \sum_{N_3\ge 0}  N_2^{-\frac{1}{56}}  \langle N_3\rangle^{-\frac{1}{7}}  \| J_x^\frac{1}{7} P_{N_2} u_2\|_{L^4_t L^\infty_x}\| J_x^\frac{1}{7} P_{N_3} u_3\|_{L^4_t L^\infty_x} \| P_{N_1} u_1 \|_{L^\infty_t H^s_x} \\
 & \lesssim &  \|u_1\|_{L^\infty_t H^s_x} \| J_x^\frac{1}{7} u_2\|_{L^4_t L^\infty_x}\| J_x^\frac{1}{7}  u_3\|_{L^4_t L^\infty_x} \; .
  \end{eqnarray*}
  {\bf 2.2.2.} $ N_1\gg N_2 $. Then we have $ |\xi_1|\sim |\xi|$ and $ |\Omega_3(\xi_1,\xi_2,\xi_3)|\sim |\xi_2+\xi_3| \xi^2 $. \\
  {\bf 2.2.2.1.} $|\xi_2+\xi_3|< |\xi|^{-1}$. Then by Plancherel and H\"older inequality,
   \begin{eqnarray*}
  I & \lesssim & \sum_{N>2^9} \sum_{0\le N_3\le N_2\ll N_1\sim N} N^{s+\frac{1}{8}} \|P_{N_1} u_1\|_{L^2_t L^2_x} N^{-1} 
  \prod_{i=2}^3 \|P_{N_i} u_i \|_{L^\infty_t L^2_x} \\
  & \lesssim &  \|u_1\|_{L^\infty_t H^s_x}  \prod_{i=2}^3 \| u_i \|_{L^\infty_t L^2_x} \
   \end{eqnarray*}
  {\bf 2.2.2.2.}  $|\xi_2+\xi_3|\ge |\xi|^{-1}$. We perform a dyadic decomposition in $ M_1\sim |\xi_2+\xi_3| $ and to evaluate the contribution for $ M_1$ and $ N\sim N_1$ fixed, we rewrite $ u_i $, $i=1,2,3$,  as
  $$
  u_i=
  Q_{\gtrsim M_1 N^2} u_i+  Q_{\ll M_1 N^2} u_i
  $$
  The contribution of all the terms that contains $Q_{\gtrsim M_1 N^2} u_1 $ can be estimated by
   \begin{eqnarray*}
  I & \lesssim & \sum_{N> 2^9}N^{s+\frac{1}{8}}\sum_{N^{-1}\lesssim M_1\ll N}\sum_{0\le N_3\le N_2\ll N  }
  \frac{M_1}{M_1 N^2} \|Q_{\gtrsim M_1 N^2} P_{\sim N} u_1\|_{X^{0,1}} \prod_{i=2}^3 \|P_{N_i} u_i \|_{L^\infty_t L^2_x} \\
   & \lesssim & \|u_1\|_{X^{s-1,1}} \prod_{i=2}^3 \| u_i \|_{L^\infty_t H^s_x} \; .
  \end{eqnarray*}
  The contributions of other terms that contain at least one projector $Q_{\gtrsim M_1 N^2}$ can be estimated in the same way thanks to \eqref{QL.1}.\\
 It remains to estimate the contribution of terms that contain only the projector $Q_{\ll M_1 N^2}$. Since $ \Omega_3\gtrsim M_1 N^2 $ and $ |\tau-\xi^3|<\frac{\xi^2}{6} $, we infer that for those terms it holds $ |\tau-\xi^3|\gtrsim M_1 N^2 $ with $ N^{-1} \lesssim M_1\lesssim 1 $. Therefore, by almost-orthogonality,
\begin{displaymath}
\begin{split}
I^2 & \lesssim  \sum_{N> 2^9} \Bigl[\sum_{N^{-1}\lesssim M_1\lesssim 1}\sum_{0\le N_3\le N_2\ll N }
  \frac{M_1 N^{s+\frac{1}{8}}}{(M_1 N^2)^\frac{1}{16}} \|Q_{\ll M_1 N^2} P_{\sim N} u_1\|_{L^2_{tx}} \\ & \quad \quad \quad \quad \quad \quad \times \prod_{i=2}^3 \|Q_{\ll M_1 N^2}  P_{N_i} u_i \|_{L^\infty_t L^2_x} \Bigr]^2\\
   & \lesssim   \prod_{i=2}^3 \| u_i \|_{L^\infty_t H^{0+}_x}^2 \sum_{N>2^9} \|P_{\sim N} u_1\|_{L^2_t H^s}^2\lesssim
     \prod_{i=1}^3 \| u_i \|_{L^\infty_t H^{s}_x}^2 \, .
\end{split}
\end{displaymath}
\end{proof}

\section{Proof of Theorem \ref{maintheo}} \label{Secmaintheo}
Fix $s>\frac13$. First it is worth noticing that we can always assume that we deal with data that have  small $ H^s $-norm.
Indeed,  if $u$ is a solution to the IVP \eqref{mKdV} on the time
interval $[0,T]$ then,  for every $0<\lambda<\infty $,
$u_{\lambda}(x,t)=\lambda u(\lambda x,\lambda^3t)$ is also a
solution to the equation in \eqref{mKdV} on the time interval
$[0,\lambda^{-3}T]$ with initial data $u_{0,\lambda}=\lambda
u_{0}(\lambda \cdot)$.  For  $\varepsilon>0 $ let us denote by $ \mathcal{B}^s(\varepsilon) $ the ball of $ H^s(\mathbb R)$, centered at the origin with radius $ \varepsilon $. Since
$$\|u_{\lambda}(\cdot,0)\|_{H^s} \lesssim\lambda^{\frac12}(1+\lambda^s)\|u_0\|_{H^s},$$  we see that we can  force $u_{0,\lambda}$  to belong to $ \mathcal{B}^s(\epsilon)$ by
choosing $\lambda \sim \min( \varepsilon^2\|u_0\|_{H^s}^{-2},1) $.
Therefore the existence and uniqueness of a solution of \eqref{mKdV} on the time interval $ [0,1] $ for small $ H^s$-initial data will ensure the existence of a unique solution $u$ to \eqref{mKdV} for arbitrary large $H^s$-initial data  on the time
interval $T\sim \lambda^3 \sim \min( \|u_0\|_{H^s}^{-6},1)$.

\subsection{Existence}
First, we begin by deriving  \textit{a priori} estimates on
smooth solutions associated to  initial data $u_0\in H^{\infty}(\mathbb R)$ that is small in $H^s(\mathbb R) $. It is known from the classical well-posedness theory that such an initial data gives rise to a global solution $u \in C(\mathbb R; H^{\infty}(\mathbb R))$ to the Cauchy problem \eqref{mKdV}.

We then deduce gathering estimates \eqref{lem-Est.1}, \eqref{lem-Est.100}, \eqref{prop-ee.1} and \eqref{trilin.1}  that
\begin{displaymath}
\|u\|_{\widetilde{L^{\infty}_T} H^s_x}^2 \lesssim \|u_0\|_{H^s}^2 \big(1+\|u_0\|_{H^s}^2 \big)^2+\|u\|_{L^{\infty}_T H^s_x}^4 \big(1+\|u\|_{L^{\infty}_T H^s_x}^2\big)^{34} \, ,
\end{displaymath}
for any $0<T \le 1$.
Moreover, observe that $\lim_{T \to 0} \|u\|_{\widetilde{L^{\infty}_T}  H^s_x}=\|u_0\|_{H^s}$. Therefore, it follows by using a continuity argument that there exists $\epsilon_0>0$ and $C_0>0$ such that
\begin{equation} \label{maintheo.2}
 \|u\|_{\widetilde{L^{\infty}_T}  H^s_x}\le  C_0\|u_0\|_{H^s} \quad \text{provided} \quad \|u_0\|_{H^s}  \le \epsilon_0 \, .
 \end{equation}

Now, let $u_1$ and $u_2$ be two solutions to the equation in \eqref{mKdV} in ${\widetilde{L^{\infty}_T}  H^s_x}$ for some $0<T\le 1$ emanating respectively from $u_1(\cdot,0)=\varphi_1$ and $u_2(\cdot,0)=\varphi_2$. We also assume that
\begin{equation} \label{maintheo.3}
 \|u_i\|_{\widetilde{L^{\infty}_T}  H^s_x}  \le C_0 \epsilon_0, \quad \text{for} \ i=1,2 \, .
\end{equation}

Let us define $w=u_1-u_2$ and $z=u_1+u_2$. Then $(w,z)$ solves
\begin{equation} \label{diffmKdV}
\left\{ \begin{array}{l}
\partial_tw+\partial_x^3w+ \frac {3\kappa}4\partial_x(z^2w)+\frac {\kappa}4 \partial_x(w^3)=0 \, , \\
\partial_tz+\partial_x^3z+\frac {\kappa}4\partial_x(z^3) + \frac{3\kappa}4\partial_x(zw^2) =0\, .
\end{array}\right.
\end{equation}
Therefore, it follows from  \eqref{lem-Est.1}, \eqref{prop-ee.1} and \eqref{trilin.1}  that $ u_1, u_2\in Y^s_T $ and
\begin{equation} \label{maintheo.4}
\|u_1-u_2\|_{L^{\infty}_TH^s_x}\lesssim  \|u_1-u_2\|_{\widetilde{L^{\infty}_T} H^s_x}\lesssim \|\varphi_1-\varphi_2\|_{H^s} \, .
\end{equation}
provided $u_1$ and $u_2$ satisfy \eqref{maintheo.3}.

\begin{remark}
Observe that no smoothness assumption on $u_1$ and $u_2$ is needed for estimate \eqref{maintheo.4} to hold. We only need $u_1$ and $u_2$ to be two weak solutions of mKdV in the sense of Definition \ref{def}, which is ensured by Remark \ref{rem2}, since $u_1$ and $u_2$ belong to $\widetilde{L^{\infty}}(]0,T[ \, : H^s(\mathbb R))$.
\end{remark}

We are going to apply \eqref{maintheo.4} to construct our solutions.
 Let $ u_0 \in H^s $ with $ s>1/3$ satisfying $\|u_0\|_{H^s}\le  \varepsilon_0$. We denote by $ u_N $ the solution of \eqref{mKdV} emanating from $ P_{\le N} u_0$ for any dyadic integer $ N\ge 1$. Since $ P_{\le N} u_0 \in H^{\infty}(\mathbb R)$, there exists a solution $u_N$ of \eqref{mKdV} satisfying
$$u_{N} \in C(\mathbb R : H^{\infty}(\mathbb R)) \quad \text{and} \quad u_{N}(\cdot,0)=P_{\le N} u_{0} \, .$$
We observe that $\|u_{0,N}\|_{H^s} \le \|u_0\|_{H^s} \le \epsilon_0$. Thus, it follows from \eqref{maintheo.2}-\eqref{maintheo.4} that for any couple of dyadic integers $ (N,M) $ with $ M<N$,
$$\|u_{N}-u_{M}\|_{\widetilde{L^{\infty}_1} H^s_x} \lesssim \|(P_{\le N}-P_{\le M})u_{0}\|_{H^s}
\underset{M \to +\infty}{\longrightarrow} 0 \, .$$
Therefore $\{u_{N}\}$ is a Cauchy sequence in $C([0,1]; H^s(\mathbb R)) \cap \widetilde{L^{\infty}}(]0,1[\, : H^s(\mathbb R))$ which converges to a solution $u \in C([0,1] ; H^s(\mathbb R)) \cap \widetilde{L^{\infty}}(]0,1[\, : H^s(\mathbb R))$ of \eqref{mKdV}. Moreover, it is clear from Propositions  \ref{se} and \ref{trilin}  that $u$ belongs to the class \eqref{maintheo.1}.
\subsection{Uniqueness}
Next, we state our uniqueness result.
\begin{lemma} \label{uniqueness}
Let $ s>\frac 13$ and
 let $u_1$ and $u_2$ be two solutions of  \eqref{mKdV} in $L^{\infty}_TH^s_x$ for some $T>0$ and satisfying $u_1(\cdot,0)=u_2(\cdot,0)=\varphi$. Then $u_1=u_2$ on $[-T,T]$.
\end{lemma}

\begin{proof} Let us define $K=\max\{\|u_1\|_{L^{\infty}_TH^s_x},\|u_2\|_{L^{\infty}_TH^s_x}\}$. Let $s'$ be a real number satisfying $\frac13<s'<s$. We get by using the uniform boundedness of $P_N$ in $L^{\infty}_TH^s_x$ that
\begin{equation} \label{uniqueness.1}
\|u_i\|_{\widetilde{L^{\infty}_T}H^{s'}_x}\lesssim  \Big(\sum_{N}N^{2(s'-s)} \Big)^{\frac12} \|u_i\|_{L^{\infty}_TH^s_x}\lesssim \|u_i\|_{L^{\infty}_TH^s_x} \, ,
\end{equation}
for $i=1,2$.

As explained above, we use the scaling property of \eqref{mKdV} and define $u_{i,\lambda}(x,t)=\lambda u_i(\lambda x,\lambda^3 t)$. Then, $u_{i,\lambda}$ are solutions to the equation in \eqref{mKdV} on the time interval $[-S,S]$ with $S=\lambda^{-3} T$ and with the same initial data $\varphi_{\lambda}=\lambda\varphi(\lambda\cdot)$. Thus, we deduce from \eqref{uniqueness.1} that
\begin{equation} \label{uniqueness.2}
\|u_{i,\lambda}\|_{\widetilde{L^{\infty}_S}H^{s'}_x} \lesssim \lambda^{\frac12}(1+\lambda^{s'})\|u_i\|_{\widetilde{L^{\infty}_T}H^{s'}_x} \lesssim \lambda^{\frac12}(1+\lambda^{s'})K, \quad \text{for} \quad i=1,2 \, .
\end{equation}
Thus, we can always choose $\lambda=\lambda>0$ small enough such that $\|u_{i,\lambda}\|_{\widetilde{L^{\infty}_S}H^{s'}_x} \le C_0 \epsilon$ with $0<\epsilon \le \epsilon_1$. Therefore, it follows from \eqref{maintheo.4} that $u_{\lambda,1} = u_{\lambda,2}$ on $[0,\min\{S,1\}]$. This concludes the proof of Lemma \ref{uniqueness} by reverting the change of variable and repeating this procedure a finite number of times.
\end{proof}

Finally, the Lipschitz bound on the flow is a consequence of estimate \eqref{maintheo.4}.

\section{\textit{A priori} estimates in $H^s$ for $s>0$} \label{Secsecondtheo}

Let $u$ be a smooth solution of \eqref{mKdV} defined in the time interval $[0,T]$ with $0<T\le 1$. Fix $0<s \le \frac13$. The aim of this section is to derive estimates for $u$ in
 the function space $ Z_T^s $ where $ Z^s $ is the Banach space endowed with the norm
 \begin{equation} \label{defZs}
 \|u\|_{Z^s}:=\|u\|_{\widetilde{L^\infty_t} H^s_x} +\|u\|_{X^{s-1,1}}  \; .
 \end{equation}
\subsection{Estimate for the $X^{s-1,1}_T$ and $L^4_TL^{\infty}_x$ norms}
\begin{proposition} \label{apriori.se}
Assume that $0<T \le 1$ and $s > 0$. Let $u\in L^\infty_T H^s_x \cap L^4_T L^\infty_x $ be a solution to \eqref{mKdV}. Then,
\begin{equation} \label{apriori.se.1}
\|u\|_{L^4_TL^{\infty}_x} \lesssim \|u\|_{L^{\infty}_TH^s_x}+\|u\|_{L^4_TL^{\infty}_x}\|u\|_{L^{\infty}_TH^s_x}^2 \, .
\end{equation}
\end{proposition}

\begin{proof}
Since $u$ is a solution to \eqref{mKdV} we use estimate \eqref{refinedStrichartz1}  with $F=\partial_x(u^3)$ and $ \delta=1+ $ to obtain
\begin{displaymath}
\begin{split}
\|u\|_{L^4_TL^{\infty}_x} &\lesssim \|u\|_{L^{\infty}_TH^{0+}_x}+\|u^3\|_{L^4_TL^1_x} \lesssim \|u\|_{L^{\infty}_TH^{0+}_x}+
\|u\|_{L^4_TL^{\infty}_x}\|u\|_{L^{\infty}_T L^2_x}^2  \, .
\end{split}
\end{displaymath}
\end{proof}

\begin{proposition} \label{apriori.triline}
Assume that $0<T \le 1$ and $s > 0$. Let $u\in \widetilde{L^\infty_T} H^s_x \cap L^4_T L^\infty_x $  be a solution to \eqref{mKdV}. Then, $u\in Z^s_T $ and
\begin{equation} \label{apriori.triline.1}
\|u\|_{Z^s_T} \lesssim \|u\|_{\widetilde{L^\infty_T} H^s_x}+\Bigl( \|u\|_{L^{\infty}_TH^{s}_x}+
\|u\|_{L^4_TL^{\infty}_x}\|u\|_{L^{\infty}_T L^2_x}^2\Bigr)^2\|u\|_{L^{\infty}_TH^s_x} \, .
\end{equation}
\end{proposition}

\begin{proof} We extend $u $ on $ \R $ by using the extension operator $ \rho_T $ defined in \eqref{defrho}. According to Lemma \ref{extension}, $ \rho_T $ is bounded, uniformly in $ 0<T<1$,  from
 $ \widetilde{L^\infty_T} H^s_x\cap X^{s-1,1}_T $ into $ Z^s $.
In view of  \eqref{apriori.se.1}, it suffices to prove that
$$
\|u\|_{X^{s-1,1}_T} \lesssim \|u_0\|_{H^s}+\|u\|_{L^4_TL^{\infty}_x}^2\|u\|_{L^{\infty}_TH^s_x} \, .
$$
This estimate can be proven in  exactly the same way  as the one of Proposition \ref{trilin}.
\end{proof}

\subsection{Integration by parts} In this Section, we will use the notations of Section \ref{Secmultest}. We also denote
$\displaystyle{m=\min_{1 \le i \neq j \le 3} |\xi_i+\xi_j|}$ and
\begin{equation} \label{m2}
A_j=\big\{(\xi_1,\xi_2,\xi_3) \in \mathbb R^3 \, : \, |\sum_{k=1\atop k \neq j}^3\xi_k|=m \big\}, \quad \text{for} \quad j=1,2,3 \, .
\end{equation}
Then, it is clear from the definition that
\begin{equation} \label{m3}
\sum_{j=1}^3 \chi_{A_j}(\xi_1,\xi_2,\xi_3)=1, \quad \textit{a.e.} \ (\xi_1,\xi_2,\xi_3) \in \mathbb R^3 \, .
\end{equation}

For $\eta\in L^\infty$, let us define the trilinear pseudo-product operator $\widetilde{\Pi}^{(j)}_{\eta,M}$ in Fourier variables by
\begin{equation} \label{def.pseudoproduct.ee.1}
\mathcal{F}_x\big(\widetilde{\Pi}^{(j)}_{\eta,M}(u_1,u_2,u_3) \big)(\xi)=\int_{\Gamma^2(\xi)}(\chi_{A_j}\eta)(\xi_1,\xi_2,\xi_3)\phi_{M}(\sum_{k=1\atop k \neq j}^3\xi_k)\prod_{l=1}^3\widehat{u}_l(\xi_l) \, .
\end{equation}
Moreover, if the functions $u_l$ are real-valued, the Plancherel identity yields
\begin{equation} \label{def.pseudoproduct.ee.2}
\int_{\mathbb R} \widetilde{\Pi}^{(j)}_{\eta,M}(u_1,u_2,u_3) \, u_4 \, dx=\int_{\Gamma^3}(\chi_{A_j}\eta)(\xi_1,\xi_2,\xi_3)\phi_{M}(\sum_{k=1\atop k \neq j}^3\xi_k) \prod_{l=1}^4 \widehat{u}_l(\xi_l) \, .
\end{equation}

Next, we derive a technical lemma involving the pseudo-products which will be useful in the derivation of the energy estimates.
\begin{lemma} \label{technical.pseudoproduct}
Let $N$ and $M$ be two homogeneous dyadic numbers satisfying $N \gg 1$. Then, for $M \ll N$, it holds
\begin{equation} \label{technical.pseudoproduct.2}
\int_{\mathbb R} P_N \widetilde{\Pi}^{(3)}_{1,M}(f_1,f_2,g) \, P_N\partial_x g \, dx = M\sum_{N_3 \sim N}\int_{\mathbb R} \widetilde{\Pi}_{\eta_3,M}^{(3)}(f_1,f_2,P_{N_3}g) \, P_Ng \, dx ,
\end{equation}
for any real-valued functions $f_1, \, f_2, \, g \in L^2(\mathbb R)$ and
where $\eta_3$ is a function of $(\xi_1,\xi_2,\xi_3)$ whose $L^{\infty}-$norm is uniformly bounded in $N$ and $M$.
\end{lemma}

\begin{proof} Let us denote by $T_{M,N}(f_1,f_2,g,g)$ the left-hand side of \eqref{technical.pseudoproduct.2}. From Plancherel's identity we have
\begin{equation*}
\begin{split}
&T_{M,N}(f_1,f_2,g,g) \\ & \quad=\int_{\mathbb R^3} \chi_{A_3}(\xi_1,\xi_2,\xi_3)\phi_M(\xi_1+\xi_2)\xi\phi_N(\xi)^2\widehat{f}_1(\xi_1)\widehat{f}_2(\xi_2)\widehat{g}(\xi_3)\overline{\widehat{g}(\xi)}d\tilde{\xi} \, ,
\end{split}
\end{equation*}
where $\xi=\xi_1+\xi_2+\xi_3$ and $d\tilde{\xi}=d\xi_1d\xi_2d\xi_3$.
We use that $\xi=\xi_1+\xi_2+\xi_3$ to decompose $T_{M,N}(f_1,f_2,g,g)$ as follows.
\begin{equation} \label{technical.pseudoproduct.3}
\begin{split}
T_{M,N}(f_1,f_2,g,g) &=M\sum_{\frac{N}2\le N_3 \le 2N}\int_{\mathbb R} \widetilde{\Pi}_{\tilde{\eta}_1,M}^{(3)}(f_1,f_2,P_{N_3}g)P_{N}gdx \\ & \quad +M\sum_{\frac{N}2\le N_3 \le 2N}\int_{\mathbb R} \widetilde{\Pi}_{\tilde{\eta}_2,M}^{(3)}(f_1,f_2,P_{N_3}g)P_{N}gdx\\ & \quad +\widetilde{T}_{M,N}(f_1,f_2,g,g)\, ,
\end{split}
\end{equation}
where
\begin{displaymath}
\tilde{\eta}_1(\xi_1,\xi_2,\xi_3)=\phi_N(\xi)\frac{\xi_1+\xi_2}M\chi_{\supp \phi_M}(\xi_1+\xi_2) \, ,
\end{displaymath}
\begin{displaymath}
\tilde{\eta}_2(\xi_1,\xi_2,\xi_3)=\frac{\phi_N(\xi)-\phi_N(\xi_3)}M\xi_3\chi_{\supp \phi_M}(\xi_1+\xi_2) \, ,
\end{displaymath}
and
\begin{displaymath}
\begin{split}
\widetilde{T}_{M,N}(f_1,f_2,g,g)=\int_{\mathbb R^3} \chi_{A_3}(\xi_1,\xi_2,\xi_3)\phi_M(\xi_1+\xi_2)\xi_3\widehat{f}_1(\xi_1)\widehat{f}_2(\xi_2)\widehat{g_N}(\xi_3)\overline{\widehat{g_N}(\xi)}d\tilde{\xi}\,
\end{split}
\end{displaymath}
with the notation $g_N=P_Ng$.

First, observe from the mean value theorem and the frequency localization that
$\tilde{\eta}_1$ and $\tilde{\eta}_2$ are
uniformly bounded in $M$ and $N$.

Next, we deal with $\widetilde{T}_{M,N}(f_1,f_2,g,g)$. By using that $\xi_3=\xi-(\xi_1+\xi_2)$ observe that
\begin{displaymath}
\begin{split}
\widetilde{T}_{M,N}(f_1,f_2,g,g)&=-\int_{\mathbb R^3}\chi_{A_3}(\xi_1,\xi_2,\xi_3)\phi_M(\xi_1+\xi_2)(\xi_1+\xi_2)\widehat{f_1}(\xi_1)\widehat{f_2}(\xi_2)\widehat{g_N}(\xi_3)\overline{\widehat{g_N}(\xi)}d\tilde{\xi}\\ & \quad +S_{M,N}(f_1,f_2,g,g)
\end{split}
\end{displaymath}
with
\begin{displaymath}
S_{M,N}(f_1,f_2,g,g)=\int_{\mathbb R^3}\chi_{A_3}(\xi_1,\xi_2,\xi_3)\phi_M(\xi_1+\xi_2)\widehat{f_1}(\xi_1)\widehat{f_2}(\xi_2)\widehat{g_{N}}(\xi_3)\xi\overline{\widehat{g_N}(\xi)}d\tilde{\xi} \, .
\end{displaymath}
Since $g$ is real-valued, we have $\overline{\widehat{g_N}(\xi)}=\widehat{g_N}(-\xi)$, so that
\begin{displaymath}
S_{M,N}(f_1,f_2,g,g)=\int_{\mathbb R^3}\chi_{A_3}(\xi_1,\xi_2,\xi_3)\phi_M(\xi_1+\xi_2)\widehat{f_1}(\xi_1)\widehat{f_2}(\xi_2)\overline{\widehat{g_{N}}(-\xi_3)}\xi\widehat{g_N}(-\xi)d\tilde{\xi} \, .
\end{displaymath}
We change variable $\hat{\xi_3}=-\xi=-(\xi_1+\xi_2+\xi_3)$, so that $-\xi_3=\xi_1+\xi_2+\hat{\xi}_3$. Thus, $S_{M,N}(f_1,f_2,g,g)$ can be rewritten as
\begin{displaymath}
-\int_{\mathbb R^3}\chi_{A_3}(\xi_1,\xi_2,-\xi_1-\xi_2-\hat{\xi}_3)\phi_M(\xi_1+\xi_2)\widehat{f_1}(\xi_1)\widehat{f_2}(\xi_2)\hat{\xi}_3\widehat{g_N}(\hat{\xi}_3)\overline{\widehat{g_{N}}(\xi_1+\xi_2+\hat{\xi}_3)}d\hat{\xi} \, ,
\end{displaymath}
where $d\hat{\xi}=d\xi_1d\xi_2d\hat{\xi}_3$. Now, observe that $|\xi_1+(-\xi_1-\xi_2-\hat{\xi}_3)|=|\xi_2+\hat{\xi}_3|$ and $|\xi_2+(-\xi_1-\xi_2-\hat{\xi}_3)|=|\xi_1+\hat{\xi}_3|$. Thus $\chi_{A_3}(\xi_1,\xi_2,-\xi_1-\xi_2-\hat{\xi}_3)=\chi_{A_3}(\xi_1,\xi_2,\hat{\xi}_3)$ and we obtain
\begin{displaymath}
S_{M,N}(f_1,f_2,g,g)=-\widetilde{T}_{M,N}(f_1,f_2,g,g) \, ,
\end{displaymath}
so that
\begin{equation} \label{technical.pseudoproduct.4}
\widetilde{T}_{M,N}(f_1,f_2,g,g)= M\int_{\mathbb R} \Pi_{\eta_2,M}^3(f_1,f_2,P_{N}g)P_Ngdx
\end{equation}
where
\begin{displaymath}
\eta_2(\xi_1,\xi_2,\xi_3)=-\frac12\frac{\xi_1+\xi_2}M\chi_{\supp \phi_M}(\xi_1+\xi_2)
\end{displaymath}
is also uniformly bounded function in $M$ and $N$.

Finally, we define $\eta_1=\tilde{\eta}_1+\tilde{\eta}_2$ and $\eta_3=\eta_1+\eta_2$. Therefore the proof of \eqref{technical.pseudoproduct.2} follows gathering \eqref{technical.pseudoproduct.3} and \eqref{technical.pseudoproduct.4}.
\end{proof}

Finally, we state a $L^2$-trilinear estimate involving the $X^{-1,1}$-norm and whose proof is similar to the one of Proposition \ref{L2trilin}.
\begin{proposition} \label{apriori.L2trilin}
Assume that $0<T \le 1$, $\eta$ is a bounded function and $u_i$ are real-valued functions in $Z^0=X^{-1,1} \cap L^{\infty}_tL^2_x$ with time support in $[0,2]$ and spatial Fourier support in $I_{N_i}$ for $i=1,\cdots 4$. Here, $N_i$ denote nonhomogeneous dyadic numbers. Assume also that $N_{max}\gg 1$, and $m=\min_{1 \le i \neq j \le 3}|\xi_i+\xi_j| \sim M \ge 1$.
Then
\begin{equation} \label{apriori.L2trilin.2}
\Big| \int_{\R \times [0,T]}\widetilde{\Pi}^{(3)}_{\eta,M}(u_1,u_2,u_3) \, u_4 \, dxdt \Big| \lesssim M^{-1} \prod_{i=1}^4(\|u_i\|_{X^{-1,1}}+\|u_i\|_{L^{\infty}_tL^2_x}) \, .
\end{equation}
Moreover, the implicit constant in estimate \eqref{apriori.L2trilin.2} only depends on the $L^{\infty}$-norm of the function $\eta$.
 \end{proposition}

\subsection{Energy estimates}
The aim of this subsection is to prove the following energy estimates for the solutions of \eqref{mKdV}.
\begin{proposition} \label{apriori.ee}
Assume that $0<T \le 1$ and $s > 0$. Let $u\in Z^s_T\cap L^4_T L^\infty_x$ be a solution to \eqref{mKdV}. Then,
\begin{equation} \label{apriori.ee.0}
\|u\|_{\widetilde{L^{\infty}_T}H^s_x}^2 \lesssim \|u_0\|_{H^s}^2+(\|u\|_{L^4_T L^\infty_x}^2+\|u\|_{Z^s_T}^2)\|u\|_{Z^s_T}^2 \, ,
\end{equation}
where $\|\cdot\|_{Z^s_T}$ is defined in \eqref{defZs}.
\end{proposition}

\begin{proof} Observe from the definition that
\begin{equation} \label{apriori.ee.1}
\|u\|_{\widetilde{L^{\infty}_T}H^s_x}^2 \sim  \sum_{N}N^{2s}\|P_Nu\|_{L^{\infty}_TL^2_x}^2
\end{equation}
Moreover, by using \eqref{mKdV}, we have
\begin{displaymath}
\frac12\frac{d}{dt}\|P_Nu(\cdot,t)\|_{L^2_x}^2 = \int_{\mathbb R} \big(P_N\partial_x(u^3)P_Nu\big) (x,t)dx \, .
\end{displaymath}
which yields after integration in time between $0$ and $t$ and summation over $N$
\begin{equation} \label{apriori.ee.3}
\|u\|_{\widetilde{L^{\infty}_T}H^s_x}^2\lesssim \|u_0\|_{H^s}^2+\sum_{N}\sup_{t \in [0,T]} \big| L_N(u)\big|\, ,
\end{equation}
where
\begin{equation} \label{apriori.ee.3.0}
L_N(u)=N^{2s} \int_{\mathbb R \times [0,t]} P_N\partial_x(u^3) \, P_Nu \, dx ds \, .
\end{equation}

In the case where $N \lesssim 1$, H\"older's inequality  leads to
\begin{equation} \label{apriori.ee.4}
\sum_{N \lesssim 1}\big| L_N(u)\big| \lesssim \|u\|_{L^4_TL^{\infty}_x}^2\|u\|_{L^{\infty}_TL^2_x}^2\lesssim   \|u\|_{L^4_TL^{\infty}_x}^2 \|u\|_{Z^s_T}^2 \, .
\end{equation}

\medskip
In the following, we can then assume that $N \gg 1$. By using the decomposition in \eqref{m3}, we get that $L_{N}(u)=\sum_{j=1}^3L_{N}^{(j)}(u)$ with
\begin{displaymath}
L_{N}^{(j)}(u)=N^{2s}\sum_{M}\int_{\mathbb R \times [0,t]}P_N \widetilde{\Pi}_{1,M}^{(j)}(u,u,u) \, P_N\partial_xu \, dx ds \, ,
\end{displaymath}
where we performed a homogeneous dyadic decomposition in $m \sim M$.
Thus, by symmetry, it is enough to estimate $L_{N}^{(3)}(u)$, that still will be denoted $L_N(u)$ for the sake of simplicity.

We decompose $L_{N}(u)$ depending on wether $M <1$, $1 \le M \ll N$ and $M \gtrsim N$. Thus
\begin{align}
L_N(u)&=N^{2s}\Big(\sum_{M \gtrsim N}+\sum_{1 \le M \ll N}+\sum_{M \le \frac12}\Big)\int_{\mathbb R \times [0,t]}P_N \widetilde{\Pi}_{1,M}^{(3)}(u,u,u) \, P_N\partial_xu \, dx ds
\nonumber \\&=:L_N^{high}(u)+L_N^{med}(u)+L_N^{low}(u) \, .  \label{apriori.ee.5}
\end{align}

\medskip

\noindent \textit{Estimate for $L_N^{high}(u)$.} Let $\tilde{u}=\rho_T(u)$ be the extension of $u$ to $\mathbb R^2$ defined in \eqref{defrho}.  Now we define $u_{N_i}=P_{N_i}\tilde{u}$, for $i=1,2,3$, $u_N=P_N\tilde{u}$ and perform dyadic decompositions in $N_i$, $i=1,2,3$, so that
\begin{displaymath}
L_N^{high}(u)=N^{2s}\sum_{M \gtrsim N}\sum_{N_1, N_2, N_3} \int_{\mathbb R \times [0,t]}P_N \widetilde{\Pi}_{1,M}^{(3)}(u_{N_1},u_{N_2},u_{N_3}) \, P_N\partial_xu \, dx ds  \, .
\end{displaymath}
Define $$\eta_{high}(\xi_1,\xi_2,\xi_3)=\frac{\xi}N\phi_N(\xi) \, .$$
It is clear that $\eta_{high}$ is uniformly bounded in $M$ and $N$. Thus, by using estimate  \eqref{apriori.L2trilin.2}, we have that
\begin{align}
\big|L_N^{high}(u)\big| &\lesssim N^{2s}\sum_{M \gtrsim N}\sum_{N_1, N_2, N_3}N\Big|\int_{\mathbb R \times [0,t]}P_N \widetilde{\Pi}_{\eta_{high},M}^{(3)}(u_{N_1},u_{N_2},u_{N_3}) \, P_N\partial_xu \, dx ds\Big| \nonumber \\ & \lesssim  N^{2s}\|u_N\|_{Z^0} \sum_{N_1,N_2,N_3}\prod_{i=1}^3\|u_{N_i}\|_{Z^0}\, , \label{apriori.ee.6}
\end{align}
since $\sum_{M \gtrsim N}N/M \lesssim 1$. Let us denote $N_{max}, \ N_{med}$ and $N_{min}$ the maximum, sub-maximum and minimum of $N_1, \ N_2, \ N_3$. It follows then from the frequency localization that $N \lesssim N_{med} \sim N_{max}$. Thus, we deduce summing \eqref{apriori.ee.6} over $N$, using the Cauchy-Schwarz inequality in $N_1, \ N_2, \ N_3$ and $N$ that
\begin{equation} \label{apriori.ee.7}
\sum_{N \gg 1}\big|L_N^{high}(u)\big|  \lesssim  \|\tilde{u}\|_{Z^s}^4 \lesssim \|u\|_{Z^s_T}^4 \, ,
\end{equation}
since $s>0$. \\

\noindent \textit{Estimate for $L_N^{med}(u)$.} To estimate $L^{med}(u)$, we decompose $\int_{\mathbb R} P_N \widetilde{\Pi}^{(3)}_{1,M}(u,u,u) \, P_N\partial_x u$ as in \eqref{technical.pseudoproduct.2}, since we are in the case $1 \le M \ll N$ and $N \gg 1$.

Once again, let $\tilde{u}=\rho_T(u)$ be the extension of $u$ to $\mathbb R^2$ defined in \eqref{defrho} and $u_{N_i}=P_{N_i}\tilde{u}$, for $i=1,2,3$, $u_N=P_N\tilde{u}$. Observe from the frequency localization that $N_3 \sim N$. We perform dyadic decompositions in $N_i$, $i=1,2,3$ and deduce from \eqref{technical.pseudoproduct.2} that
\begin{displaymath}
\big|L_N^{med}(u)\big| \lesssim N^{2s}\sum_{1 \le M \ll N}\sum_{N_1, N_2}\sum_{N_3 \sim N}M\Big|\int_{\mathbb R \times [0,t]}P_N \widetilde{\Pi}_{\eta_3,M}^{(3)}(u_{N_1},u_{N_2},u_{N_3}) \, P_N\partial_xu \, dx ds\Big| \, ,
\end{displaymath}
where $\eta_3$\footnote{see the proof of Lemma \ref{technical.pseudoproduct} for a definition of $\eta_3$.} is uniformly bounded  in the range of summation of $M, \, N, \, N_1, \, N_2$ and $N_3$. Then, we deduce from \eqref{apriori.L2trilin.2}  that
\begin{equation} \label{apriori.ee.10}
\big|L_N^{med}(u)\big| \lesssim \sum_{1 \le  M \ll N}\sum_{N_1,N_2}\sum_{N_3 \sim N}\|u_{N_1}\|_{Z^0}\|u_{N_2}\|_{Z^0}\|u_{N_3}\|_{Z^s}\|u_N\|_{Z^s} \, .
\end{equation}
Observe that $\max\{N_1,N_2 \} \gtrsim M$. Therefore, we deduce after summing \eqref{apriori.ee.10} over $N \sim N_3 \gg 1$, $N_1$, $N_2$ and $M$ that
\begin{equation} \label{apriori.ee.11}
\sum_{N \gg 1}\big|L_N^{med}(u)\big|  \lesssim  \|\tilde{u}\|_{Z^s}^4 \lesssim \|u\|_{Z^s_T}^4 \, ,
\end{equation}
since $s>0$. Note that in the last step we used that $ \|\tilde{u}\|_{Z^s}^2 \sim \sum_{N} \|\tilde{u}_N\|_{Z^s}^2 $.

\medskip

\noindent \textit{Estimate for $L_N^{low}$.} In this case, we also have $N \gg 1$ and $M \ll N$. Thus the decomposition in \eqref{technical.pseudoproduct.2} yields
\begin{displaymath}
L_N^{low}(u)=N^{2s}\sum_{M \le \frac12}M\sum_{N_3 \sim N}\int_{\mathbb R \times [0,t]}\widetilde{\Pi}_{\eta_3,M}^{(3)}(u,u,P_{N_3}u)P_Nu \, dxds  \, ,
\end{displaymath}
where $\eta_3$ is defined in the proof of Lemma \ref{technical.pseudoproduct}. Since $\eta_3$ is uniformly bounded in $N$ and $M$, we deduce from \eqref{prod4-est} and H\"older's inequality in time (recall here that $0<t \le T \le 1$) that
\begin{displaymath}
\begin{split}
\big|L_N^{low}(u) \big| &
 \lesssim N^{2s}\sum_{M \le 1/2}M^2\|u\|_{L^{\infty}_TL^2_x}^2\sum_{N_3 \sim N}\|P_{N_3}u\|_{L^{\infty}_TL^2_x} \|P_Nu\|_{L^{\infty}_TL^2_x} \, .
\end{split}
\end{displaymath}
Thus, we infer that
\begin{equation} \label{apriori.ee.12}
\sum_{N \gg 1}\big|L_N^{low}(u)\big|  \lesssim   \|u\|_{L^{\infty}_TL^2_x}^2\|u\|_{\widetilde{L^{\infty}_T}H^s_x}^2 \lesssim \|u\|_{Z^s_T}^4 \, .
\end{equation}

Finally, we conclude the proof of estimate \eqref{apriori.ee.0} gathering \eqref{apriori.ee.3}, \eqref{apriori.ee.4}, \eqref{apriori.ee.5}, \eqref{apriori.ee.7}, \eqref{apriori.ee.11} and \eqref{apriori.ee.12}.
\end{proof}

\subsection{Proof of Theorem \ref{secondtheo}}  By using a scaling argument as in Section \ref{Secmaintheo}, it suffices to prove Theorem \ref{secondtheo} in the case where the initial datum $u_0$ belongs to $ H^\infty(\R) \cap \mathcal{B}^s(\epsilon_0)$, where $  \mathcal{B}^s(\epsilon_0)$ is the ball of $H^s$ centered at the origin and of radius $\epsilon_0$.  Let $ u $ be the smooth solution emanating from $ u_0 $ .
Setting  $\Gamma^s_T(u)=\|u\|_{\widetilde{L^\infty_T}H^s_x}+\|u\|_{L^4_TL^{\infty}_x}$, it follows gathering \eqref{apriori.triline.1}, \eqref{apriori.se.1} and \eqref{apriori.ee.0} that
\begin{displaymath}
\Gamma^s_T(u) \lesssim \|u_0\|_{H^s}+\Gamma^s_T(u)^2+\Gamma^s_T(u)^{14} \, .
\end{displaymath}
 Observe that $\lim_{T \to 0} \Gamma^s_T(u)=c\|u_0\|_{H^s}$. Therefore, it follows by using a continuity argument that there exists $\epsilon_0>0$ such that
\begin{displaymath}
\Gamma^s_T(u) \lesssim \|u_0\|_{H^s} \quad \text{provided} \quad \|u_0\|_{H^s} \le \epsilon_0 \, .
\end{displaymath}
Moreover, \eqref{apriori.triline.1} ensures that
$$ \|u\|_{X^{s-1,1}_T} \lesssim  \|u_0\|_{H^s} \; . $$
Now, assume that  $ u_0\in H^s(\R) $ with $ \|u_0\|_{H^s}\le \varepsilon_0/2 $. We approximate $ u_0 $ by a sequence of smooth initial data $\{u_{0,n}\} \subset H^\infty(\R) $  such that $ \|u_{0,n} \|_{H^s} \le \varepsilon_0 $.
By passing to the limit on the sequence of emanating smooth solutions, the above  \textit{a priori} estimate ensures the existence of a solution of
 \eqref{mKdV}  for $ s>0$ in the sense of Definition \ref{def}. This solution  belongs to $ \widetilde{L^\infty_T}H^s_x\cap L^4_T L^\infty_x\cap X^{s-1,1}_T  \hookrightarrow L^3_{Tx}$.
  Note that, since $ s>0$,  there is no difficulty to pass to the limit on the nonlinear term by a compactness argument.
This concludes the proof of Theorem \ref{secondtheo} .

\medskip
\noindent \textbf{Acknowledgments.}
The authors are very grateful to the  anonymous Referee who  pointed out an error in a previous  version of this work and greatly improved the present version with numerous helpful suggestions and comments.
D.P. would like to thank the L.M.P.T. at Universit\'e Fran\c cois Rabelais for the kind hospitality during the elaboration of this work. He is also grateful to Gustavo Ponce for pointing out the reference \cite{ChHoTa}.

\bibliographystyle{amsplain}

\end{document}